\documentclass[a4paper,11pt]{article}
\usepackage[utf8]{inputenc}

\usepackage{boxedminipage}
\usepackage{amsfonts}
\usepackage{amsmath} 
\usepackage{amssymb}
\usepackage{graphicx}
\usepackage{amsthm}
\usepackage{t1enc}
\usepackage{subfig}
\usepackage{cancel}
 \usepackage{textcmds}
 \usepackage{mathabx}
 \usepackage{verbatim}

\newtheorem{theorem}{Theorem}[section]
\newtheorem{lemma}[theorem]{Lemma}

\newtheorem{definition}[theorem]{Definition}

\newtheorem{question}{Question}
\newtheorem*{theorem*}{Theorem}
\newtheorem{claim}{Claim}

\newcommand{\forceP}{\mathbb{P}}
\newcommand{\forceQ}{\mathbb{Q}}

\newcommand{\ZFC}{\mathsf{ZFC}}

\newcommand{\ZFP}{\mathsf{ZF}^-}

\newcommand{\CH}{\mathsf{CH}}

\newcommand{\PD}{\mathsf{PD}}

\def\undertilde#1{\mathord{\vtop{\ialign{##\crcr
$\hfil\displaystyle{#1}\hfil$\crcr\noalign{\kern1.5pt\nointerlineskip}
$\hfil\tilde{}\hfil$\crcr\noalign{\kern1.5pt}}}}}

\title{A Failure of $\Pi^1_{n+3}$-Reduction in the Presence of $\Sigma^1_{n+3}$-Separation}
\author{ Stefan Hoffelner\footnote{This research was funded in parts by the Austrian Science Fund (FWF) Grant-DOI 10.55776/P37228. Additional support by the Deutsche Forschungsgemeinschaft (DFG German Research Foundation) under Germany's Excellence Strategy EXC 2044 390685587, Mathematics M\"unster: Dynamics-Geometry-Structure. For the purpose of open access, the authors have applied a CC BY public copyright license to any Author Accepted Manuscript version arising from this submission.}  }

\begin{document}

\maketitle

\begin{abstract}
We show that one can force over $L$ that $\Sigma^1_3$-separation holds, while $\Pi^1_3$-reduction fails, thus separating these two principles for the first time. The construction can be lifted to canonical inner models $M_n$ with $n$-many Woodin cardinals, yielding that, assuming the existence of $M_n$,  $\Sigma^1_{n+3}$-separation can hold, yet $\Pi^1_{n+3}$-reduction fails.
\end{abstract}

\section{Introduction}

Descriptive Set Theory serves as a fundamental framework for investigating the structure and properties of sets of real numbers. Two central concepts within this theory, the Separation Property, introduced in the early 1920s and the Reduction Property, introduced by Kuratowski in the mid-1930's, have garnered significant attention due to their profound implications for properties of projective subsets of the real numbers.

\begin{definition}
 We say that a projective pointclass $\Gamma \in \{ \bf{\Sigma}^1_n$ $ \mid n \in \omega\} \cup \{\bf{\Pi}^1_n$ $ \mid n \in \omega \}$ $\cup \,\{ {\Sigma}^1_n$ $ \mid n \in \omega\} \cup \{{\Pi}^1_n$ $ \mid n \in \omega \}$ has the separation 
 property (or just separation) iff every pair $A_0$ and $A_1$ of disjoint elements of  $\Gamma$ has a separating set $C$, i.e. a set $C$ such that $A_0 \subset C$ and $A_1 \subset \omega^{\omega} \setminus C$ and such that $C
   \in \Gamma \cap \check{\Gamma}$, where  $\check{\Gamma}$ denotes the dual pointclass of $\Gamma$.
\end{definition}

\begin{definition}
We say that a projective pointclass $\Gamma \in \{ \bf{\Sigma}^1_n$ $ \mid n \in \omega\} \cup \{\bf{\Pi}^1_n$ $ \mid n \in \omega \}$ $\cup \, \{ {\Sigma}^1_n$ $ \mid n \in \omega\} \cup \{{\Pi}^1_n$ $ \mid n \in \omega \}$ satisfies the $\Gamma$-reduction property (or just reduction) if every pair $B_0,B_1$ of $\Gamma$-subsets of the reals can be reduced by a pair of $\Gamma$-sets $R_0,R_1$, which means that $R_0 \subset B_0$, $R_1 \subset B_1$, $R_0 \cap R_1= \emptyset$ and $R_0 \cup R_1=B_0 \cup B_1$.
\end{definition}

It follows immediately from the definitions that $\Gamma$-reduction implies $\check{\Gamma}$-separation. It is very natural to ask whether the reverse direction is also true.

Since their introduction, many results have been proved which shed light on how  separation and reduction can behave among the projective pointclasses. These results can be obtained using two very different set theoretic assumptions, which draw very different scenarios of the properties of separation and reduction.

The first assumption is $V=L$, or rather the existence of a $\Sigma^1_2$-definable, good projective well-order of the reals. Recall that a good $\Sigma^1_2$-definable well-order is a $\Sigma^1_2$-well-order with the additional property that 
also the relation $\operatorname{InSeg} (x,y) :\Leftrightarrow \{ (x)_i \mid i \in \omega \} = \{z  \mid z <_L y \}$ is $\Sigma^1_2$, where $(x)_i$ denotes some recursive decoding of $x$ into $\omega$-many reals, $(x)_i$ being the $i$-th real decoded out of $r$. By results of J. Addison (\cite{Addison}) the existence of a good $\Sigma^1_n$-well-order implies $\Sigma^1_m$-uniformization for every $m \ge n$. As $\Sigma^1_m$-uniformization implies $\Sigma^1_m$ reduction, the assumption of $V=L$ implies that for every $n \ge 1$, $\Sigma^1_n$-reduction and $\Pi^1_n$-separation is true (the case $n=1$ follows from Kondo's theorem that $\Pi^1_1$-uniformization is true).

The second assumption which settles the behaviour of reduction and separation on the projective hierarchy is projective determinacy ($\mathsf{PD}$). By the results of Y. Moschovakis (\cite{Moschovakis2}), under $\mathsf{PD}$, for every $n \in \omega$, $\Pi^1_{2n+1}$ and $\Sigma^1_{2n+2}$ sets have the scale property, which in particular implies that $\Pi^1_{2n+1}$ and $\Sigma^1_{2n+2}$ sets have the uniformization property and so $\Pi^1_{2n+1}$ and $\Sigma^1_{2n+2}$ reduction is true. By the famous theorems of D. Martin and J. Steel (see \cite{MS}) on the one hand, and H. Woodin (see \cite {MSW}) on the other hand, determinacy assumptions on projective sets and large cardinal assumptions are two sides of the very same coin.

Note that under $V=L$ and also under $\PD$, $\Gamma$ separation holds because already the stronger $\check{\Gamma}$-reduction (in fact $\check{\Gamma}$-uniformization) holds. So these results do not 
shed light on the question stated above, whether $\Gamma$-separation and $\check{\Gamma}$-reduction are different properties at all. A partial answer to the question was first given by R. Sami in his PhD thesis from 1976 \cite{Sami}. In it he showed (among other interesting results) that
after adding a single Cohen real to $L$, the resulting universe will satisfy that
$\Pi^1_3$-separation holds, yet $\Sigma^1_3$-reduction fails. And additionally $\Sigma^1_n$-reduction holds again for $n \ge 4$. His results inspired L. Harrington to produce a model in which $\bf{\Pi}^1_3$-separation holds but there
is a (lightface) $\Sigma^1_3$ set, which can not be reduced by any pair of $\bf{\Sigma}^1_3$-sets, thus $\Sigma^1_3$-reduction fails (a write up of Harrington's proof can be found in \cite{KL}).

The question for the other side of the projective hierarchy, namely for $\Sigma^1_3$-separation and $\Pi^1_3$-reduction remained open though since then.

The goal of this article is to produce the counterpart to L. Harrington's result:
\begin{theorem*}
One can force over $L$ a model of $\bf{\Sigma}^1_3$-separation over which there is a pair of $\Pi^1_3$-sets, which can not be reduced by any pair of $\bf{\Pi}^1_3$-sets.
\end{theorem*}
We add that our result uses completely different techniques and methods than Sami's and Harrington's theorems.
It follows from a modification of the arguments from \cite{Ho2}, that the proofs can be transferred from $L$ to $M_n$, the canonical inner model with $n$-many Woodin cardinals.

\begin{theorem*}
Assuming that $M_n$ exists, there is a model of $\bf{\Sigma}^1_{n+3}$-separation over which there is a pair of $\Pi^1_{n+3}$-sets, which can not be reduced by any pair of $\bf{\Pi}^1_{n+3}$-sets.

\end{theorem*}

The proof relies on the construction and ideas from \cite{Ho1}, where a universe with the $\bf{\Sigma}^1_3$-separation property is produced via forcing over $L$. However, we will introduce some simplifications of the original argument, yielding a cleaner presentation. We will use the coding machinery from \cite{Ho3} which is basically the same as in \cite{Ho2}, though we opted to present it slightly differently to reduce the potential of confusion. In \cite{Ho3} a universe is forced over $L$, where $\Pi^1_3$-reduction holds but the stronger $\Pi^1_3$-uniformization fails. Our proof presented here will take a quite different direction though and uses a more direct diagonalization argument, where we actively work towards two $\Pi^1_3$-sets $B_0, B_1$ which can not be reduced by a pair of $\bf{\Pi}^1_3$-sets. This idea of actively building a pair of projective sets which witness the failure of separation is further pursued, in a considerably easier setting in \cite{HOFFELNER2026103667}, where a universe is built where separation fails on both sides of the third projective level, in the presence of a weaker form of Martin's axiom.  As we simultaneously have to work towards a stronger failure of $\Pi^1_3$-reduction, we need to substantially alter the original definitions and arguments for forcing the $\bf{\Sigma}^1_3$-separation. The underlying theme in all these forcing constructions is to use a coding machine which, to some extent, is closed under taking products. This approach is in some or another form necessary as became clear in hindsight with the construction presented in \cite{HOFFELNER2025110272}, which forces the $\Sigma^1_{n+2}$-uniformization property for all $n$ simultaneously and is very flexible. There is a second flexible method producing just $\Sigma^1_3$-uniformization (see \cite{HOFFELNER2024103466}), which can be used to show that, somewhat surprisingly on first sight, $\mathsf{BPFA}$ and $\aleph_1=\aleph_1^L$ together outright imply $\Sigma^1_3$-uniformization. Thus any way to force $\Sigma^1_3$-separation must use a technique which can not be married with the ones from \cite{HOFFELNER2025110272} and \cite{HOFFELNER2024103466}. Being closed under products neatly destroys any possibility of such a marriage.

\subsection{Notation}
The notation we use will be mostly standard, we hope. Diverging from standard conventions we write $\forceP=(\forceP_{\alpha} \, : \, \alpha < \gamma)$ for a forcing iteration of length $\gamma$ with initial segments $\forceP_{\alpha}$. The $\alpha$-th factor of the iteration will be denoted with $\forceP(\alpha)$, this is nonstandard as typically one writes $\dot{\forceQ}_{\alpha}$. Note here that we drop the dot on $\forceP(\alpha)$, even though $\forceP(\alpha)$ is in fact a $\forceP_{\alpha}$-name of a partial order. 
If $\alpha' < \alpha < \gamma$, then we write $\forceP_{\alpha' ,\alpha}$ to denote the intermediate forcing of $\forceP$ which happens in the interval $[\alpha',\alpha)$, i.e. $\forceP_{\alpha', \alpha}$ is such that 
$\forceP \cong \forceP_{\alpha'} \ast \forceP_{\alpha', \alpha}$.

We write $\Sigma_n(X)$, for $X$ an arbitrary set, to denote the set of formulas which are $\Sigma_n$ and use elements from $X$ as a parameter.

We write $1 \Vdash \varphi$ whenever every condition in $\forceP$ forces $\varphi$, and make deliberate use of restricting partial orders below conditions, that is, if $p \in \forceP $ is such that $p \Vdash \varphi$, we let $\forceP':= \forceP_{\le p}:=\{ q \in \forceP \, : \, q \le p\}$ and use $\forceP'$ instead of $\forceP$. This is supposed to reduce the notational load of some definitions and arguments.

\section{Independent Suslin trees in $L$, almost disjoint coding}
The coding method of our choice utilizes Suslin trees, which can be generically destroyed in an independent way of each other. 
\begin{definition}
 Let $\vec{T} = (T_{\alpha} \, : \, \alpha < \kappa)$ be a sequence of Suslin trees. We say that the sequence is an 
 independent family of Suslin trees if for every finite set of pairwise distinct indices $e= \{e_0, e_1,...,e_n\} \subset \kappa$ the product $T_{e_0} \times T_{e_1} \times \cdot \cdot \cdot \times T_{e_n}$ 
 is a Suslin tree again.
\end{definition}
Note that an independent sequence of Suslin trees $\vec{T} = (T_{\alpha} \, : \, \alpha < \kappa)$ has the property that if $A \subset \kappa$ and we form 
$ \prod_{i \in A} T_i $
with finite support, where each $T_i$ denotes the forcing we obtain if we force with the nodes of the tree  as conditions using the tree order as the partial order, then in the resulting generic extension $V[G]$, for every $\alpha \notin A$, $V[G] \models `` T_{\alpha}$ is a Suslin tree$"$. To see this, we assume the opposite, namely there is an $\alpha \notin A$ such that $V[G]$ thinks that $T_{\alpha}$ is not Suslin anymore. We note that $V[G]$ is a generic extension of $V$ obtained with a forcing with the ccc (see \cite{Ho}, Lemma 51 for a more general argument of this type). So there is an $\aleph_1$-sized antichain of $T_{\alpha}$ in $V[G]$. But then the finitely supported product of the form $(\prod_{i \in A}  T_i )\times T_{\alpha}$ can not have the ccc, which is a contradiction to the  assumed independence of $\vec{T}$.

One can easily force the existence of independent sequences of Suslin trees with products of Jech's or Tennenbaum's forcing, or with just products of ordinary Cohen forcing. On the other hand independent sequences of length $\omega_1$ already exist in $L$. 

\begin{theorem}\label{DefinitionIndependentSequence}
Assume that $\aleph_1= \aleph_1^L$ and that $ (M,\in)$ is a transitive, $\omega_1$-containing, uncountable model of $\ZFC^{-} + ``\aleph_1$ exists$"$.  Then there is an independent sequence $\vec{T} = ( T_{\alpha} \mid \alpha < \omega_1)$ of $L$-Suslin trees and the sequence $\vec{T}$ is uniformly $\Sigma_1 ( \{ \omega_1 \} )$-definable over $M$. To be more precise,
there is a $\Sigma_1$-formula $\phi$  with $\omega_1$ as the unique parameter, which does not depend on the model $M$, such that
the relation $\{ (t,\gamma,\eta) \mid \gamma,\eta < \omega_1 \land t \in T_{\eta}^{\gamma}  \}$, where $T_{\eta}^{\gamma}$ denotes the $\gamma$-th level of $T_{\eta} \in \vec{T}$, is definable over $M$ using $\phi$. In particular $\vec{T}$ can be defined correctly in any transitive, uncountable model $M$ of $\ZFC^-$ and ``$\aleph_1$ exists$"$ using $\phi$ and $\omega_1$ as a parameter of $\phi$.

\end{theorem}

\begin{proof}
We first argue in $L$ and later argue that our to be presented construction relativizes to any $M$ as in the proposition. We fix the canonical $\diamondsuit$-sequence $(a_{\alpha} \subset \alpha \mid \alpha< \omega_1)$ in $L$. 
Next we alter the usual construction of a Suslin tree from $\diamondsuit$ to construct an $\omega_1$-sequence of Suslin trees $\vec{T}= (T_{\alpha} \mid \alpha < \omega_1)$. 

We fix a $\Delta_1$-definable bijection of the set of pairs of countable ordinals with $\omega_1$. We use this bijection in the usual way to code sets of pairs of countable ordinals with sets of countable ordinals. Under this identification, trees will be coded as subsets of ordinals.

We assume that $\alpha< \omega_1$ is a limit ordinal, and split the set $a_{\alpha} \subset \alpha$ into two parts via considering the even and odd ordinals of $a_{\alpha}$. We assume first that if we recursively split again the even ordinals of $a_{\alpha}$ into three sets $A,B,C \subset \operatorname{Even} (a_{\alpha})$, then $A$ codes a well-order of order type $\beta < \alpha < \omega_1$, $B$ codes a well-order of order type $\gamma <\alpha$ and $C$ codes a finite subset of $\omega_1$. This indicates that we aim to define the $\beta$-th level of $T_{\gamma}$ as follows.

We assume first that $\beta$ is a limit ordinal and $T^i_{\gamma}$ and $T^i_{\eta}$ for $\eta \in C$ has been defined already for each $i < \beta$. Now we consider the odd ordinals of $a_{\alpha}$ and assume that $\operatorname{Odd} (a_{\alpha})$ codes a set $A \subset \bigcup_{i < \beta } T_{\gamma}^i \times \prod_{\eta \in C} T_{\eta}^i$ which is an antichain there.

We define the $\beta$-th level $T^{\beta}_{\gamma}$ of $T_{\gamma}$ such that it seals off the antichain $A$. To be more specific we choose $ T_{\gamma}^{\beta} \times \prod_{\eta \in C} T_{\eta}^{\beta}$ in such a way that $A$ remains a maximal antichain in all further extensions of $ T_{\gamma}^{\beta} \times \prod_{\eta \in C} T_{\eta}^{\beta}$.

Otherwise we just extend $T^{\beta}_{\gamma}$ via adding top nodes on countably many branches through $T^{\beta}_{\gamma}$.

If $\beta$ is not a limit ordinal we define $T^{\beta}_{\gamma}$ under the assumption that $T^{\beta-1}_{\gamma}$ is already defined and let $T^{\beta}_{\gamma}$ be just the result of putting $\omega$-many new nodes above each node of $T^{\beta-1}_{\gamma}$.

We let $T_{\gamma}: = \bigcup_{\beta < \omega_1} T^{\beta}_{\gamma}$ and claim that 
$(T_{\gamma} \mid \gamma < \omega_1)$ is an independent sequence of Suslin trees in $L$.

Indeed, if $A \in L$ is an antichain in some $\prod_{\gamma \in e} T_{\gamma}$, then there is a club in $L$ of ordinals $\alpha$ such that $A \cap \alpha= a_{\alpha}$ and $\operatorname{E ven}(a_{\alpha})$ codes $A,B,C$ where $A$ codes $\beta$, $B$ codes $\gamma$ and $C$ codes $e \setminus \{\gamma\}$ and $\operatorname{Odd} (a_{\alpha})$ codes a maximal antichain in $\bigcup_{i < \beta } T_{\gamma}^i \times \prod_{\eta \in C} T_{\eta}^i$. But then $\operatorname{Odd} (a_{\alpha})$ got sealed off and $\prod_{\gamma \in e} T_{\gamma}$ has no uncountable antichains, so is a Suslin tree in $L$.

The definability of $\vec{S}$ comes from the fact that the canonical $\diamondsuit$-sequence in $L$ is $\Sigma_1 (\{ \omega_1 \})$-definable. We can use $L_{\omega_1}$ (which is $\Sigma_1( \{\omega_1 \} )$ to correctly define $\diamondsuit$ over it and consequentially $\vec{S}$ becomes definable over $L_{\omega_1}$ as well. The above considerations can be simulated correctly already in any transitive, uncountable $M$ which models $\ZFC^{-}$, as it will compute $L_{\omega_1}$ correctly and the rest of the construction is performed inside the latter model. 
\end{proof}

We remark, that the above proof shows that countable levels of countable initial segments of $\vec{S}$ are already uniformly definable over countable transitive models $M$ of $\ZFP$ and $``\aleph_1$ exists and $\aleph_1= \aleph_1^L"$.
That is, if $M$ is as just stated and $\gamma, \eta<\aleph_1^M$ and $t \in T^{\gamma}$ such that $\phi(t,\gamma,\omega_1)$ is true in $L$, then $M \models \phi (t \upharpoonright \omega_1^M,\gamma,\omega_1)$ and vice versa. Indeed if $M$ is such a model, and $ M \cap \omega_1=\aleph_1^M= (\aleph_1^L)^M$, then
the canonical $\diamondsuit$-sequence can be computed correctly up to $\aleph_1^M$ inside $M$'s version of $L$. So the whole construction of $\vec{S}$ inside $L_{\omega_1}$, when repeated inside $L_{\aleph_1^M}$ will produce the claimed initial segment of $\vec{S}$ restricted to trees of height $\aleph_1^M$.

The trees from $\vec{S}$ will later be used to define two countable families of $\Sigma^1_3$-sets, which will be manipulated in such a way that they will witness $\bf{\Sigma}^1_3$-separation and the failure of $\Pi^1_3$-reduction. As we need two such families, we split $\vec{S}$ into two independent sequences
\[ \vec{S}^0 := \{ S_{\alpha} \mid \alpha \text{ is even}\} \]
and
\[\vec{S}^1:= \{ S_{\alpha} \mid \alpha \text{ is odd.} \} \]

We briefly introduce the almost disjoint coding forcing due to R. Jensen and R. Solovay. We will identify subsets of $\omega$ with their characteristic function and will use the word reals for elements of $2^{\omega}$ and subsets of $\omega$ respectively.
Let $D=\{d_{\alpha} \, : \, \alpha < \aleph_1 \}$ be a family of almost disjoint subsets of $\omega$, i.e. a family such that if $r, s \in D$ then 
$r \cap s$ is finite. Let $X\subset  \omega$  be a set of ordinals. Then there 
is a ccc forcing, the almost disjoint coding $\mathbb{A}_D(X)$ which adds 
a new real $x$ which codes $X$ relative to the family $D$ in the following way
$$\alpha \in X \text{ if and only if } x \cap d_{\alpha} \text{ is finite.}$$
\begin{definition}\label{definitionadcoding}
 The almost disjoint coding $\mathbb{A}_D(X)$ relative to an almost disjoint family $D$ consists of
 conditions $(r, R) \in [\omega]^{<\omega} \times D^{<\omega}$ and
 $(s,S) < (r,R)$ holds if and only if
 \begin{enumerate}
  \item $r \subset s$ and $R \subset S$.
  \item If $\alpha \in X$ and $d_{\alpha} \in R$ then $r \cap d_{\alpha} = s \cap d_{\alpha}$.
 \end{enumerate}
\end{definition}
We shall briefly discuss the $L$-definable, $\aleph_1^L$-sized almost disjoint family of reals $D$  we will use throughout this article. The family $D$ is the canonical almost disjoint family one obtains when recursively adding the $<_L$-least real $x_{\beta}$ not yet chosen and replace it with $d_{\beta} \subset \omega$ where this $d_{\beta}$  is the real which codes the initial segments of $x_{\beta}$ using some recursive bijections between $\omega$ and $\omega^{<\omega}$. The definition of $D$ is uniform over any uncountable, transitive $\ZFP$-models $M$ with, as we can correctly compute $L$ up to $\aleph_1^L$ inside $M$ and then apply the above definition inside $L$'s version of $M$. Even more is true, if $M$ is a countable, transitive model of $\ZFP+``$ $\aleph_1$ exists and $\aleph_1=\aleph_1^L"$, then $M$ will compute $D \upharpoonright \omega_1^M$ in a correct way. The reason is again, that $M$ can define an initial segment of $L$ correctly which suffices to calculate $D \upharpoonright \omega_1^M$.

Last we state a short lemma which will be helpful when showing that our coding forcings work the way they should.
\begin{lemma}\label{a.d.coding preserves Suslin trees}
 Let $T$ be a Suslin tree and let $\mathbb{A}_D(X)$ be the almost disjoint coding which codes
 a subset $X$ of $\omega_1$ into a real with the help of an almost disjoint family
 of reals $D$ of size $\aleph_1$. Then $$\mathbb{A}_{D}(X) \Vdash_{} T \text{ is Suslin }$$
 holds.
\end{lemma}
\begin{proof}
 This is clear as $\mathbb{A}_{D}(X)$ has the Knaster property, thus the product $\mathbb{A}_{D}(X) \times T$ is ccc and $T$ must be Suslin in $V^{\mathbb{A}_{D}(X)}$. 
\end{proof}

\section{Coding machinery}
We continue with the construction of the appropriate notions of forcing which we want to use in our proof. The goal is to first define a coding forcing $\operatorname{Code} (x)$ for reals $x$, which will force for $x$ that a certain $\Sigma^1_3$-formula $\Phi(x)$ becomes true in the resulting generic extension. The coding method is basically the same as in \cite{Ho2} and \cite{Ho3}.

In a first step we force over $L$ to destroy all members of $\vec{S}= \vec{S}^0 \cup \vec{S}^1$ via generically adding an $\omega_1$-branch, that is we  form \[\forceP^0:=\prod_{\alpha \in \omega_1} S_{\alpha}\] with finite support. Note that this is an $\aleph_1$-sized, ccc forcing over $L$.

In a second step
we add $\omega_1$-many $\omega_1$-Cohen subsets with a countably supported product, for technical reasons we want the second forcing to be defined as in $L$ and not as in $L^{\forceP_0}$. We let $$\forceP^1:= (\prod_{\alpha< \omega_1} \mathbb{C} (\omega_1))^L.$$
Note that this forcing is  $\sigma$-closed only over $L$ so we need to argue that the two step iteration $\forceP_0 \ast \forceP_1$ preserves $\aleph_1$. Forcing with the two step iteration $\forceP_0 \ast \forceP_1$ is the same as forcing with the product $\forceP_0 \times \forceP_1$ as the latter sits densely in the former. As $\forceP^0 \times \forceP^1$ is isomorphic to $\forceP^1 \times \forceP^0$ we can look at the latter and note that $\forceP^1$ is $\sigma$-closed over $L$, so in particular it does not add new $\omega_1$-branches to trees from $L$ \footnote{Indeed any name of an $\omega_1$-length branch $\dot{b}$ through some tree $T$ gives rise to an $\omega_1$ branch through $T$ in $L$ using the $\sigma$-closure of the forcing.}. Thus $\vec{T}$ remains an independent sequence of Suslin trees in $L^{\forceP_1}$ and $\forceP_0$ is a ccc forcing in $L^{\forceP_1}$. So $\aleph_1$ is preserved when forcing with $\forceP^0 \times \forceP^1$ and moreover and $\CH$ remains true.

 We use $W$ to denote this generic extension of $L$, we fix a filter $G^0 \times G^1$ which is generic for $\forceP^0 \times \forceP^1$  and let \[W=L[G^0][G^1]. \] 

 Let $x \in W$ be a real,  and let $m,k \in \omega$ and let $\eta <\omega_1$. We simply write $(x,m,k)$ for a real $w$ which codes the triple $(x,m,k)$ in a recursive way. The forcing $\operatorname{Code}(x,m,k,1,\eta)$ \footnote{The other coding forcing,  $\operatorname{Code}(x,m,k,0,\eta)$, which uses the trees from the $\vec{S}^0$-sequence instead is defined in the analogous way}  which codes the triple $(x,m,k)$ into $\vec{S^1}$ is defined as the almost disjoint coding forcing of a specific set $Y \subset \omega_1$, that is
 \[ \operatorname{Code}(x,m,k,1) := {\mathbb{A}}({Y}).\]  We will define the crucial set $Y \subset \omega_1$ now.

 To ease notation we let $g \subset \omega_1$ be $g_{\eta}$ for $\eta < \omega_1$, where $g_{\eta}$ is the $\eta$-th coordinate of the $\prod_{\alpha < \omega_1} \mathbb{C}(\omega_1)$-generic filter over $L^{\forceP^0}$. We let $\rho: ([\omega_1]^{\omega})^{L} \rightarrow \omega_1$ be some canonically definable, constructible bijection between
these two sets. We use $\rho$ and $g$ to define the set $h \subset \omega_1$, which eventually shall be the set of indices of $\omega$-blocks of $\vec{S}$, where we ``code up the characteristic function of the real $(x,m,k)"$, the latter slogan will be made precise in a moment. Let \[h:= \{\rho( g \cap \alpha) \,: \, \alpha < \omega_1 \}\] and let
\begin{align*}
A:= &\{ \omega \gamma +2n \mid \gamma \in h, n \notin (x,m,k) \} \cup \\& \{\omega \gamma + 2n+1 \mid \gamma \in h, n \in (x,m,k) \}.
\end{align*}

Let $X \subset \omega_1$ be chosen such that it canonically codes the following objects:
\begin{itemize}
\item The set $A \subset \omega_1$.
\item The set $\{ b_{\beta} \subset S^1_{\beta} \mid \beta \in A \}$ of generic $\omega_1$-branches we added with $\forceP_0$. 

\end{itemize}

Note that, when working in $L[X]$ and if $\gamma \in h$ then
 we can read off $(x,m,k)$, and thus  we say that $(x,m,k)$ is coded into $\vec{S}^1$ at the $\omega$-block starting at $\gamma$,  via looking at the $\omega$-block of $\vec{S^1}$-trees starting at $\gamma$ and determine which tree has an $\omega_1$-branch in $L[X]$.

\begin{align*}
   (\ast)_1({\gamma},(x,m,k),) = {} & \parbox[t]{0.75\textwidth}{$n \in (x,m,k)$ if and only if $S^1_{\omega \cdot \gamma +2n+1}$ has an $\omega_1$-branch, and $n \notin (x,m,k)$ if and only if $S^1_{\omega \cdot \gamma +2n}$ has an $\omega_1$-branch..}
\end{align*}

Indeed if $n \notin (x,m,k)$ then we added a cofinal branch through $S^1_{\omega \cdot \gamma+ 2n}$. If on the other hand $S^1_{\omega \cdot\gamma +2n}$ does not have an $\omega_1$-branch in $L[X]$ then we must have added an $\omega_1$-branch through $S^1_{\omega \cdot \gamma +2n+1}$ as we always add an $\omega_1$-branch through either $S^1_{\omega \cdot \gamma +2n+1}$ or $S^1_{\omega \cdot \gamma +2n}$ and adding branches through some $S^1_{\alpha}$'s  will not affect that some $S^1_{\beta}$ remain Suslin in $L[X]$, as $\vec{S}^1$ is independent.

We note that we can apply an argument resembling David's trick \footnote{see \cite{David} for the original argument, where the strings in Jensen's coding machinery are altered such that certain unwanted universes are destroyed. This destruction is emulated in our context as seen below.} in this situation. We rewrite the information of $X \subset \omega_1$ as a subset $Y \subset \omega_1$ using the following line of reasoning.
Keeping lemma \ref{DefinitionIndependentSequence} in mind, it is clear that any transitive, $\aleph_1$-sized model $N$ of $\ZFP$ which contains $X$ as an element will be able to first define $\vec{S}^1$ correctly  and also correctly decode out of $X$ all the information regarding $(x,m,k)$ being coded at each $\omega$-block of $\vec{S}^1$ starting at every $\gamma \in h$. 
Consequently, if we code the model $(N,\in)$ as a set $X_N \subset \omega_1$, then for any uncountable $\beta$ such that $L_{\beta}[X_N] \models \ZFP$:
\begin{align*}
    L_{\beta}[X_N] \models & \text{\ldq The model decoded out of }X_N \text{ satisfies $(\ast)_{1}(\gamma,(x,m,k))$} \\& \text{for every $\gamma \in h$\rdq.}
\end{align*}

For the latter assertion to be meaningful, we need that the model decoded out of $X_N$ has $h$ as an element but this is clear as $h$ can be read off from $A$ which in turn is coded into $X$ which belongs to $N$ as an element.
We proceed by noting that there will be an $\aleph_1$-sized ordinal $\beta$ as above and we can fix a club $C \subset \omega_1$ and a sequence $(M_{\alpha} \, : \, \alpha \in C)$ of countable elementary submodels  of $L_{\beta} [X_N]$ such that
\[\forall \alpha \in C (M_{\alpha} \prec L_{\beta}[X_N] \land M_{\alpha} \cap \omega_1 = \alpha)\]
Now let the set $Y\subset \omega_1$ code the pair $(C, X_N)$ such that the odd entries of $Y$ should code $X_N$ and if $E(Y)$ denotes  the set of even entries of $Y$ and $\{c_{\alpha} \, : \, \alpha < \omega_1\}$ is the enumeration of $C$ then
\begin{enumerate}
\item $E(Y) \cap \omega$ codes a well-ordering of type $c_0$.
\item $E(Y) \cap [\omega, c_0) = \emptyset$.
\item For all $\beta$, $E(Y) \cap [c_{\beta}, c_{\beta} + \omega)$ codes a well-ordering of type $c_{\beta+1}$.
\item For all $\beta$, $E(Y) \cap [c_{\beta}+\omega, c_{\beta+1})= \emptyset$.
\end{enumerate}
We claim that the following assertion is true, which can be seen as a local version of $(\ast)(\gamma, (x,m,k))$.

\begin{align*}
\sigma_1 (x,m,k) := {} & \parbox[t]{0.75\textwidth}{For any countable transitive model $M$ of ``$\ZFP$ and $\aleph_1$ exists'' such that $\omega_1^M=(\omega_1^L)^M$ and $Y \cap \omega_1^M \in M$, $M$ can construct its version of the universe $L[Y \cap \omega_1^M]$, and the latter will see that there is an $\aleph_1^M$-sized, transitive model $\bar{N} \in L[Y \cap \omega_1^M]$, $\bar{N} \models \ZFC^- \land \omega_1^M \in \bar{N}$ which models $(\ast) ({\gamma,(x,m,k)})$ for $\aleph_1^M$-many $\gamma$.}
\end{align*}
Indeed, if $M$ is as asserted by $\sigma(x,m,k)$, then in particular $Y \cap \omega_1^M \in M$  and $\omega_1^M \in C$ by property 3 of $Y$. Thus $\omega_1^M= \omega_1^{M_{\alpha}}$ for an $M_{\alpha} \prec L_{\beta}[X_N]$ and the set which gets decoded out of $Y \cap \omega_1^M$ is the same no matter when we decode inside $M_{\alpha}$
or $M$. By elementarity, if $M_{\alpha}$ defines its version of $L[Y \cap \omega_1^M]$ internally then this inner model will see that there is a transitive $\bar{N} \models \ZFC^-$, $\omega_1^M  \in \bar{N} $ which models $(\ast) (\gamma,(x,m,k))$ for $\aleph_1^M$-many $\gamma$. So by absoluteness of the decoding, $M$ will find this $\bar{N}$ in its version of $L[Y \cap \omega_1^M]$ as well and again by absoluteness $\bar{N}$ satisfies the desired properties also when working in $M$, thus $\sigma(x,m,k)$ is true.  

We have finally defined the desired set $Y$ and now we use 
\[\operatorname{Code} (x,m,k,1,\eta):= \mathbb{A} (Y) \]
 relative to our previously defined, almost disjoint family of reals $D \in  L $ (see the paragraph after Definition 2.5)  to code the set $Y$ into one real $r$. This forcing only depends on the subset of $\omega_1$ we code, thus $\mathbb{A}_D(Y)$ will be independent of the surrounding universe in which we define it, as long as it has the right $\omega_1$ and contains the set $Y$.

The effect of the coding forcing $\operatorname{Code} (x,m,k,1,\eta)$ is that it generically adds a real $r$ such that
this formula becomes true:

\begin{align*}
\Psi_1 (r, (x,m,k)) := {} & \parbox[t]{0.75\textwidth}{For any countable, transitive model $M$ of ``$\ZFC^-$ and $\aleph_1$ exists'', such that $\omega_1^M=(\omega_1^L)^M$ and $r \in M$, $M$ can construct its version of $L[r]$, denoted by $L[r]^M$, which in turn thinks that there is a transitive $\ZFC^-$-model $\bar{N}$ of size $\aleph_1^M$ such that $\bar{N}$ believes $(\ast)({\gamma,(x,m,k))}$ for an $\aleph_1^M$-sized set of ordinals $\gamma$.}
\end{align*}

Indeed, if $r$ and $M$ are as above, then $M$ and $L[r]^M$ will compute the almost disjoint family $D$ up to the real indexed with $\omega_1 \cap M$ correctly, as discussed below the definition 2.3. As a consequence, $L[r]^M$ will contain the set $Y \cap \omega_1^M$, where $Y \subset \omega_1$ is as in the statement of $\sigma(x,m,k)$. So, as argued above,  in $L[Y \cap \omega_1^M]$, there is an $\aleph_1^M$-sized, transitive $\bar{N}$ which models $(\ast)({\gamma},(x,m,k))$ for every $\gamma \in h \cap M$, as claimed.

Note that $ \Psi_1 (r, (x,m,k))$ is a $\Pi^1_2$-formula in the parameters $r$ and $(x,m,k)$, as the set $h \cap M \subset \omega_1^M$ is coded into $r$. We say in the above situation that the real $(x,m,k)$ \emph{ is written into $\vec{S}^1$}, or that $(x,m,k)$ \emph{is coded into} $\vec{S^1}$. To summarize our discussion, given an arbitrary real of the form $(x,m,k)$, then our forcing $\operatorname{Code} (x,m,k,1)$, when applied over $W$, will add a real $r$ which will turn the $\Pi^1_2$-formula $ \Psi_1 (r, (x,m,k))$ into a true statement in $W^{\operatorname{Code} (x,m,k,1)}$.

The coding forcing which codes a given real $(x,m,k)$ into the $\vec{S}^0$, denoted by $\operatorname{Code} {(x,m,k,0,\eta)}$ is defined in the same way and so are the according formulas $\sigma_0$ and $\Psi_0$.

 The projective and local statement $ \Psi_1 (r, (x,m,k))$, if true,  will determine how certain inner models of the surrounding universe will look like with respect to branches through $\vec{S}$.
That is to say, if we assume that $ \Psi_1 (r, (x,m,k))$ holds. Then $r$ also witnesses the truth of $ \Psi_1 (r, (x,m,k))$ for any transitive  model $M$ of  the theory ``$\ZFP+$ $\aleph_1$ exists and $\aleph_1= \aleph_1^L"$,  which contains $r$ (i.e. we can drop the assumption on the countability of $M$).
Indeed if we assume 
that there would be an uncountable, transitive $M$, $r \in M$, which witnesses that $ \Psi_1 (r, (x,m,k))$ is false. Then by L\"owenheim-Skolem, there would be a countable $N\prec M$, $r\in N$ which we can transitively collapse to obtain the transitive $\bar{N}$. But $\bar{N}$ would witness that $ Psi_1 (r, (x,m,k))$ is not true for every countable, transitive model, which is a contradiction.

Consequently, the real $r$ carries enough information that
the universe $L[r]$ will see that certain trees from $\vec{S}^1$ have branches in that
\begin{align*}
n \in (x,m,k) \Rightarrow L[r] \models  ``S^1_{\omega \gamma + 2n+1} \text{ has an $\omega_1$-branch}".
\end{align*}
and
\begin{align*}
n \notin (x,m,k) \Rightarrow L[r] \models ``S^1_{\omega \gamma + 2n} \text{ has an $\omega_1$-branch}".
\end{align*}
Indeed, the universe $L[r]$ will see that there is a transitive model $N$ of ``$\ZFP+$ $\aleph_1$ exists and $\aleph_1=\aleph_1^L"$ which believes $(\ast)$ for every $\gamma \in h \subset \omega_1$, the latter being coded into $r$. But by upwards $\Sigma_1$-absoluteness, and the fact that $N$ can compute $\vec{S}^1$ correctly, if $N$ thinks that some tree in $\vec{S^1}$ has a branch, then $L[r]$ must think so as well.

\section{Allowable forcings}

Next we define the set of forcings which we will use in our proof.
We aim to iterate the coding forcings we defined in the last section.  

\begin{definition}\label{def:allowable}
Let $W=L[G^0 \times G^1]$ be our ground model. Let $\alpha < \omega_1$ and let $F\in L$, $F: \alpha \rightarrow L$ be a bookkeeping function.
A finite support iteration $\forceP=(\forceP_{\beta}\,:\, {\beta< \alpha})$ is called allowable (relative to the bookkeeping function $F$) if the function $F$ determines $\forceP$ inductively as follows:
 \begin{itemize}
 \item[] We assume that $\beta \ge 0$ and $\forceP_{\beta}$ is defined.
 We let $G_{\beta}$ be a $\forceP_{\beta}$-generic filter over $W$ and assume that $F(\beta)=(\dot{x},\dot{m},\dot{k}, \dot{l}, \dot{\eta})$ is a tuple of $(\forceP^0 \times\forceP^1) \ast \forceP_{\beta}$-names. We assume that $\dot{x}^{(G^0 \times G^1)\ast G_{\beta}}=:x$ is a real, $\dot{m}^{(G^0 \times G^1)\ast G_{\beta}}=:m $ and $\dot{k}^{(G^0 \times G^1)\ast G_{\beta}}=:k$ are natural numbers, $\dot{l}^{(G^0 \times G^1)\ast G_{\beta}}=:l \in \{0,1\}$ and $\dot{\eta}^{(G^0 \times G^1)\ast G_{\beta}}=:\eta$ is an ordinal $< \omega_1$. 
 
 \paragraph{The Freshness Condition:} We say that the \emph{Freshness Condition} fails at stage $\beta$ if there is a prior stage $\gamma < \beta$ and a $\forceP_{\gamma}$-name of a tuple $(\dot{a}', \dot{m}', \dot{k}', \dot{l}' ,\dot{\eta}') $ such that $\dot{a}'^{(G^0 \times G^1)\ast G_{\gamma}}= a \in \omega^{\omega}$, $\dot{m}'^{(G^0 \times G^1)\ast G_{\gamma}}= m' \in \omega$, $\dot{k}'^{(G^0 \times G^1)\ast G_{\gamma}}=k' \in \omega$, $\dot{l}'^{(G^0 \times G^1)\ast G_{\gamma}}=l' \in \{0,1\}$, and $\dot{\eta}'^{(G^0 \times G^1)\ast G_{\gamma}} = \eta$, and the forcing at stage $\gamma$ was defined as $\forceP(\gamma)^{G_{\gamma}} = \operatorname{Code}(a, m', k', l', \eta)$. In other words, the condition fails if the coding area $\eta$ has already been used for coding at an earlier stage of the iteration.
 
 Then we split into two cases:
 \begin{itemize}
 \item If the Freshness Condition fails at stage $\beta$, then we force with the trivial forcing.
 \item If the Freshness Condition holds at stage $\beta$, then let $\forceP(\beta)^{G_{\beta}}:= \operatorname{Code} (x,m,k,l,\eta)$.
 \end{itemize}
 \end{itemize}
\end{definition}

If $\forceP \in W$ is a forcing such that there is an $\alpha < \omega_1$ and an $F \in L$, $F: \alpha \rightarrow L$ such that $\forceP$ is allowable with respect to $F$, then we often just drop the $F$ and simply say that $\forceP \in W$ is allowable. 

As allowable forcings form the base set of an inductively defined shrinking process, they are also denoted by 0-allowable with respect to $F$ to emphasize this fact.
Informally speaking, the bookkeeping $F$ hands us at every step a real of the form $(x,m,k)$, decides whether to use $\vec{S}^0$ or $\vec{S}^1$ and uses a $\mathbb{C}(\omega_1)$-set (this set we will call a \emph{coding area}) which gives rise to a subset of $h \subset \omega_1$ which corresponds to the places where we code up the relevant branches through $\vec{S}^l$ to compute $(x,m,k)$ using the coding mechanism described in the previous section. The definition ensures that the sets $h$ determined by the coding areas are an almost disjoint family of $\omega_1$, i.e. two distinct such sets have countable intersection. The definition of allowable demands that each such coding area is used at most once in an allowable forcing, which will help to ensure that we will not accidentally code an unwanted $(x,m,k)$ into $\vec{S}$. 

\begin{definition}
Let $\forceP= ((\forceP_{\alpha}, \forceP({\alpha})) \mid \alpha < \delta)$ be an allowable forcing. Let $G \subset \forceP$ be a generic filter over $W$. Then
\begin{align*}
C^G:= \{ \eta < \omega_1\mid \exists \beta < \delta \exists \dot{x},&\dot{m},\dot{k}, \dot{l}, \dot{\eta} \in W^{\forceP_{\beta}}  \\&( {\forceP}({\beta}))^{G_{\beta}} = \operatorname{Code} (\dot{x}^{G_{\beta}},\dot{m}^{G_{\beta}},\dot{k}^{G_{\beta}},\dot{l}^{G_{\beta}},\dot{\eta}^{G_{\beta}}=\eta )\}
\end{align*}
is the set of coding areas of $\forceP$ relative to $G$. 
We also let 
\[ C^{\forceP} := \{\eta < \omega_1 \mid \exists p \in \forceP ( p \Vdash \eta  \in C^{\dot{G}} \}. \]
\end{definition}
It is immediate from the definition that $C^G$ and also $C^{\forceP}$ are always countable sets for every allowable $\forceP$. Next we derive some properties of allowable forcings.
\begin{definition}\label{def:allowable_extension}
Let $\mathbb{P} = (\mathbb{P}_{\xi} : \xi < \delta_1)$ be an allowable ($0$-allowable) forcing over $W$ with bookkeeping function $F_{\mathbb{P}} : \delta_1 \to L$, and let $G_{\mathbb{P}}$ be a $\mathbb{P}$-generic filter over $W$. A finite support iteration $\mathbb{Q} = (\mathbb{Q}_{\beta} : \beta < \delta_2)$ in $W[G_{\mathbb{P}}]$ is called \emph{allowable over the extension $W[G_{\mathbb{P}}]$} if there exists a bookkeeping function $F_{\mathbb{Q}} : \delta_2 \to L$ such that at every stage $\beta < \delta_2$, if we let $G_{\mathbb{Q}_{\beta}}$ denote the $\forceQ_{\beta}$-generic filter over $W[G_{\forceP}]$, then the iterand $\mathbb{Q}(\beta)^{G_{\mathbb{Q}_\beta}}$ is determined exactly as in Definition \ref{def:allowable} using $F_{\mathbb{Q}}(\beta)$, but subject to the \emph{Global Freshness Condition}: 
\begin{quote}
We say the \emph{Global Freshness Condition} fails at stage $\beta$ if the proposed coding area $\eta$ has already been used at a prior stage in $\mathbb{Q}$ (i.e., some $\alpha < \beta$) or it was already used at any stage in the forcing $\mathbb{P}$ (i.e., some $\xi < \delta_1$).
\end{quote}
If the Global Freshness Condition fails, $\mathbb{Q}(\beta)^{G_{\mathbb{Q}_\beta}}$ is defined to be the trivial forcing. Otherwise, we force with the corresponding coding forcing determined by $F_{\forceQ} (\beta)$ and the generic filters $G_{\forceP}$ and $G_{\forceQ_{\beta}}$.
\end{definition}

\begin{lemma}\label{iteration_of_allowable}
Let $\mathbb{P}$ be an allowable forcing over $W$ of length $\delta_1$, and suppose $\mathbb{P} \Vdash ``\dot{\mathbb{Q}} \text{ is an allowable forcing over } W[\dot{G}_{\mathbb{P}}] \text{ of length } \delta_2"$. Then the two-step iteration $\mathbb{P} \ast \dot{\mathbb{Q}}$ is an allowable forcing over $W$ of length $\delta_1 + \delta_2$.
\end{lemma}

\begin{proof}
Fix a generic filter $G_{\forceP}$. Let $F_1 : \delta_1 \to L$ be the bookkeeping function for $\mathbb{P}$, and let $F_2 : \delta_2 \to L$ be the bookkeeping function witnessing that $\dot{\mathbb{Q}}$ is allowable over $W[G_{\mathbb{P}}]$. 

We define the combined bookkeeping function $F : \delta_1 + \delta_2 \to L$ by:
\begin{align*}
    F(\xi) &= 
    \begin{cases} 
        F_1(\xi) & \text{if } \xi < \delta_1 \\
        F_2(\beta) & \text{if } \xi = \delta_1 + \beta \text{ for some } \beta < \delta_2 
    \end{cases}
\end{align*}

Let $\mathbb{R} = (\mathbb{R}_{\xi} : \xi < \delta_1 + \delta_2)$ be the finite support iteration generated by $F$ over $W$. We show $\mathbb{R} \cong \mathbb{P} \ast \dot{\mathbb{Q}}$ by induction on $\xi$.

For $\xi < \delta_1$, $F \restriction \delta_1 = F_1$, yielding $\mathbb{R}_{\xi} = \mathbb{P}_{\xi}$.

For $\xi = \delta_1 + \beta$, the iterand $\mathbb{R}(\xi)$ is determined by $F(\delta_1 + \beta) = F_2(\beta)$, the generic filters $G_{\forceP}$, $G_{\forceQ_{\beta}}$ and the Freshness Condition evaluated over the interval $[0, \delta_1 + \beta)$. This interval naturally partitions into $[0, \delta_1)$ (the stages of $\mathbb{P}$) and $[\delta_1, \delta_1 + \beta)$.

Working in $W[G_{\forceP}] [G_{\forceQ_{\beta}}]$, by Definition \ref{def:allowable_extension}, the Global Freshness Condition for $\dot{\mathbb{Q}}$ at stage $\beta$ holds if and only if the proposed coding area $\eta$ satisfies $\eta \notin C^{G_{\mathbb{P}}} \cup C^{G_{\mathbb{Q}_\beta}}$. This is equivalent to the standard Freshness Condition for $\mathbb{R}$ evaluated at stage $\delta_1 + \beta$. 

Consequently, the generic iterands coincide: $\mathbb{R}(\delta_1 + \beta)^{G_{\mathbb{R}}} = \dot{\mathbb{Q}}(\beta)^{G_{\mathbb{Q}}}$. We conclude $\mathbb{R} \cong \mathbb{P} \ast \dot{\mathbb{Q}}$, proving that the two-step iteration is allowable over $W$.
\end{proof}

 \begin{lemma} \label{FirstPropertiesOfAllowableForcings}

\begin{enumerate}
\item If $\forceP=(\forceP(\beta) \, : \, \beta < \delta) \in W$ is allowable then for every $\beta < \delta$, $\forceP_{\beta} \Vdash| \forceP(\beta)|= \aleph_1$, thus every factor of $\forceP$ is forced to have size $\aleph_1$.
\item Every allowable forcing over $W$ is ccc and thus preserves cardinals.
\item Every allowable forcing over $W$ preserves $\CH$. Furthermore, if $\forceP= (\forceP(\alpha) \, : \, \alpha < \omega_1) \in W$ is an $\omega_1$-length iteration such that each initial segment of the iteration is allowable over $W$, then $W^{\forceP} \models \CH$.
\item The product of two allowable forcings $\forceP$ and $\forceQ$ can be densely embedded into an allowable forcing provided that $C^{\forceP} \cap C^{\forceQ}=\emptyset$.
\end{enumerate}
\end{lemma}
\begin{proof}
The first, the second and the third assertion follow immediately from the definition modulo some well-known results.

The fourth item follows from the fact that almost disjoint coding forcings consist of finite conditions, thus their definition is absolute in every universe as long as the universe contains the subset of $\omega_1$ the almost disjoint coding forcing wants to code. As a consequence the forcing $\forceP \times \forceQ$ embeds densely into the finite support iteration of $\forceP \ast \check{\forceQ}$\footnote{
Let $i: \forceP \times \forceQ \to \forceP \ast \check{\forceQ}$ be the canonical embedding defined by concatenation. To show that $i$ has a dense image, fix an arbitrary condition $r \in \forceP \ast \check{\forceQ}$. Because the iteration $\forceP \ast \check{\forceQ}$ uses finite support, the tail $r \restriction [\delta_1, \delta_1+\beta)$ has only finitely many non-trivial coordinates. We can iteratively extend the head $r \restriction \alpha \in \forceP$ finitely many times to obtain a condition $p_1 \le r \restriction \alpha$ in $\forceP$ that completely decides the ground-model values of the tail at these coordinates. These decided values naturally define a condition $p_2 \in \forceQ$ (which has finite support since $r$ does), yielding a ground-model pair $(p_1, p_2)$ such that $i(p_1, p_2) \le r$.}. We claim that the iteration $\forceP \ast \check{\forceQ}$ is allowable. If $F_0: \delta_0 \rightarrow L$ is the bookkeeping function which witnesses that $\forceP$ is allowable, and $F_1: \delta_1 \rightarrow L$ witnesses that $\forceQ$ is allowable, then we define $F$ with domain $\delta_0 + \delta_1$ as follows. We set $F(\beta) := $ $F_0(\beta)$ if $\beta < \delta_0$ and $F(\delta_0 + \beta):= F_1 (\beta)$. The bookkeeping $F$ witnesses that $\forceP \ast \check{\forceQ}$ is allowable. This is trivial for $\beta< \delta_0$. And if we are at a stage $F(\delta_0 + \beta)$, then we will force with a coding forcing if and only if we use the same coding forcing when defining $\forceQ$ over $W$ with the help of $F_1$ at stage $\beta$, which uses the fact that  $C^{\forceP} \cap C^{\forceQ} = \emptyset$.

\end{proof}
Every allowable forcing is determined by the associated $F$ which is list of names of reals. We shall utilize this fact in showing that every allowable forcing can in fact be defined already in a proper inner model of $W$.

\begin{lemma}\label{DefinabilityInProperInnerModels}
    Let $\forceP \in W$ be an allowable forcing and let $F: \delta \rightarrow L$ be its bookkeeping. Then there is an uncountable, co-uncountable subset $I \subset \omega_1$ and a countable subset $J \subset \omega_1$ such that
    $\forceP$ can successfully be defined already in an inner model of $W$ of the form
    $L[(G^0 \upharpoonright I)  \times (G^1 \upharpoonright J)]$, where $G^0 \upharpoonright I$ is just the restriction of the generic $G^0$ to coordinates which are in I and $G^1 \upharpoonright J$ being defined similarly.
\end{lemma}
\begin{proof}
    The proof is via induction on the length $\delta$ of the iteration.
    Assume first that $\delta=1$ then we can assume that $\forceP = \operatorname{Code} (x)$ for some real $x \in W$. Recall that $W$ is defined as the generic extension of $L$ via $\forceP^0$ which generically adds branches to each tree in $\vec{S}$ and $\forceP^1= (\prod_{i \omega_1} \mathbb{C} (\omega_1))^L$. The reals in $W= L[G^0 \times G^1]$ are all elements of $L[G^0]$ already as is immediate from Easton's Lemma (see see Lemma 15.19 from \cite{Jech}). So in particular there is a countable $I \subset \omega_1$ such that $w \in L[G^0 \upharpoonright I]$ which shows the lemma for one step iterations.

    Now we assume that the lemma is true for allowable forcings of length $\delta$ and our goal is to show that it also must be true for allowable forcings of length $\delta+1$.
    We fix an allowable forcing $\forceP$ of length $\delta+1$ and write $\forceP= \forceP_{\delta} \ast \operatorname{Code} (\dot{w})$ for some $(\forceP^0 \times \forceP^1)\ast \forceP_{\delta}$-name of a real $\dot{w}$. By our induction hypothesis $\forceP_{\delta}$ is already definable in some inner model $L[G^0 \upharpoonright I] [G^1 \upharpoonright J]$. So the name $\dot{w}$ can be written as a $[(\forceP^0 \upharpoonright I) \times (\forceP^1 \upharpoonright J) \ast \forceP_{\delta}] \times [(\forceP^0 \upharpoonright (\omega_1 \setminus I)) \times \forceP^1 \upharpoonright (\omega_1 \setminus J)]$-name. Again by Easton's Lemma, this time applied over the ground model $L[G^1 \upharpoonright J]$, we obtain that the name $\dot{w}$ is in fact (equivalent to) a   $[(\forceP^0 \upharpoonright I) \times (\forceP^1 \upharpoonright J) \ast \forceP_{\delta}] \times (\forceP^0 \upharpoonright (\omega_1 \setminus I))$-name. Note that the forcing $\forceP^0 \upharpoonright (\omega_1 \setminus I)$ is an $\omega_1$-length iteration with finite support which adds a branch through every Suslin tree indexed in the set $\omega_1 \setminus I$ so by standard facts of forcing theory the name $\dot{w}$ is in fact (equivalent to) a $(\forceP^0 \upharpoonright I) \times (\forceP^1 \upharpoonright J) \ast \forceP_{\delta} \times (\forceP^0 \upharpoonright \tilde{I}))$-name, for a countable set $\tilde{I}$. We add $\tilde{I}$ to $I$ and finally note that the coding forcing $\operatorname{Code}(\dot{w})$, as it is ccc will potentially only use a countable set of coding areas $\tilde{J}$. So $\forceP= \forceP_{\delta} \ast \operatorname{Code} (\dot{w})$ can successfully be defined in the inner model
    $L[G^0 \upharpoonright I \cup \tilde{I}][G^1 \upharpoonright J \cup \tilde{J}]$ which shows the successor case.

    The limit case follows immediately from the above via taking countable unions. 
    \end{proof}

Let $\forceP= (\forceP(\beta) \, : \, \beta < \delta)$ be an allowable forcing with respect to some $F \in W$.
The set of  (names of) reals which are enumerated by $F$, and where the second case in the definition of allowable applies, i.e. we actually used a coding forcing, is dubbed \emph{the set of reals which are coded by $\forceP$}. That is, for every $\beta$, if we let $\dot{x}_{\beta}$ be the (name) of a real  listed by $F(\beta)$ and if we let $G \subset \forceP$ be a generic filter over $W$ and finally if we let
$ \dot{x}_{\beta}^G =:x_{\beta}$,  then we say that
$\{ x_{\beta} \, : \, \beta < \delta \}$ is the set of reals coded by $\forceP$ and $G$ (though we will suppress the $G$).

Next we show, that iterations of 0-allowable forcings will not add accidentally new elements to the set of reals defined by the $\Sigma^1_3$-formula \[\Phi_1((x,m,k)):= \exists r \Psi_1 (r,(x,m,k)), \]
where $\Psi_1$ is defined in the previous section and by the same argument will also not add non-intended members to the set defined by 
\[ \Phi_0 := \exists r \Psi_0( r,(x,m,k)).\] We let
\[ \Phi(x) \equiv \Phi_0 (x) \lor \Phi_1 (x). \]
To utilize already established jargon we will say that $``x$ is coded into $\vec{S^i}$'' whenever $\Phi_i (x)$ is true. Likewise, if $\lnot \Phi_i (x)$ holds true, then we say that ``$x$ is not coded into $\vec{S^i}$''.

\begin{lemma}\label{nounwantedcodes}
If $\forceP \in W$ is allowable, $\forceP=(\forceP_{\beta} \, : \, \beta < \delta)$, $G \subset \forceP$ is generic over $W$ and $\{ x_{\beta} \, : \, \beta < \delta\}$ is the set of reals which are coded by $\forceP$. Let $\Phi(v_0)$ be the distinguished formula from above. Then
in $W[G]$, the set of reals which satisfy $\Phi(v_0)$ is exactly 
$\{ x_{\beta} \, : \, \beta < \delta\}$.
\end{lemma}
\begin{proof}
Let $G$ be $\forceP$ generic over $W$, we work in $W[G]$. Let $g= (g_{\beta} \, : \, {\beta} < \delta)$ be the set of the $\delta$-many coding areas of $\forceP$ relative to $G$. We let $\rho : ([\omega_1]^{\omega})^L \rightarrow \omega_1$ be our fixed, constructible bijection and let $h_{\beta}= \{ \rho (g_{\beta} \cap \alpha) \, : \, \alpha < \delta\}$. Note that the family $\{h_{\beta} \,: \, \beta < \delta \}$ forms an almost disjoint family of subsets of $\omega_1$, that is two distinct members of  $\{h_{\beta} \,: \, \beta < \delta \}$ always have countable intersection. Thus there is an $\eta < \omega_1$ such that for arbitrary distinct $\beta_1$, $\beta_2 < \delta$,  $h_{\beta_1}\cap h_{\beta_2} \subset \eta$.

We assume for a contradiction, that there is a real $(x,m,k)$ which satisfies $\Psi_1 ((x,m,k))$ but $(x,m,k) \notin \{x_{\beta} \mid \beta < \delta \}$.
As a consequence, by the discussion right at the end of the third section, $r$ codes an unbounded set $h \subset \omega_1$ such that for every $\gamma \in h$ the following holds true.
\begin{align*}
n \in (x,m,k) \Rightarrow L[r] \models  ``S^1_{\omega \gamma + 2n+1} \text{ has an $\omega_1$-branch}".
\end{align*}
and
\begin{align*}
n \notin (x,m,k) \Rightarrow L[r] \models ``S^1_{\omega \gamma + 2n} \text{ has an $\omega_1$-branch}".
\end{align*}

As $(x,m,k)$ is distinct from every $x_{\beta}$, there must be $\aleph_1$-many $\alpha > \eta$ and an $n \in \omega$ such that, without loss of generality,
\[L[r] \models ``S^1_{\omega \alpha + 2n+1} \text{ has an $\omega_1$-branch}",\]
yet for every $r_{\beta}$ witnessing that $\Psi_1(r_{\beta}, x_{\beta})$ holds true
\[L[r_{\beta}] \models  ``S^1_{\omega \alpha + 2n+1} \text{ does not have an $\omega_1$-branch}".\]

Invoking the (proof of the) last lemma there is an uncountable, co-uncountable $I \subset \omega_1$ and 
a countable $J \subset \omega_1$ such that $\forceP \in L[G^0 \upharpoonright I][G^1 \upharpoonright J]$ and the set $I$ consists of all the indices of trees determined by the coding areas $G^1 \upharpoonright J$ with the addition of only countably many indices from $\vec{S}$. We denote this countable set of indices with $C$.
We collect the sets \begin{align*}
A_{\beta}:= &\{ \omega \gamma +2n \mid \gamma \in h_{\beta}, n \notin x_{\beta} \} \cup \\& \{\omega \gamma + 2n+1 \mid \gamma \in h_{\beta}, n \in x_{\beta} \}.
\end{align*}
As $G \subset \forceP$ picks at each stage $\beta < \delta$ of its iteration exactly one coding area $g_{\beta}$, the generic extension \[L[G^0 \upharpoonright I]  [G^1 \upharpoonright J] [ G]\] can be re-written as
\[ L [\{G^0_{\xi} \mid \xi \in \bigcup_{\beta < \delta}  A_{\beta} \} ][\{G^0_{\xi} \mid \xi \in C\}][(g_{\beta} \mid \beta  < \delta)] [\{ G^0_{\xi } \mid \xi \in I \land \xi \notin (C \cup \bigcup_{\beta <\delta} A_{\beta} ) \}].\]

In particular each tree from $\vec{S}$ with index not in  $C \cup \bigcup_{\beta <\delta} A_{\beta}$ is still Suslin in the inner model $L [\{G^0_{\xi} \mid \exists \beta < \delta (\xi \in A_{\beta} \} ][\{G^i_{\xi} \mid \xi \in C\}][(g_{\beta} \mid \beta  < \delta)]$ by the independence of the sequence.

It is therefore possible to fix an index $\alpha \notin C \cup \bigcup_{\beta <\delta} A_{\beta}$ such that there is a real $r$ with \[L[r] \models ``S^1_{\omega \alpha + 2n+1} \text{ has an $\omega_1$-branch}"\] whereas for every $\beta < \delta$
\[L[r_{\beta}] \models  ``S^1_{\omega \alpha + 2n+1} \text{ does not have an $\omega_1$-branch}".\]

We claim however  that there is no real in $W[G]$ such that $W[G] \models L[r] \models ``S^1_{\omega \alpha + 2n+1}$ has an $\omega_1$-branch$"$, which will be the desired contradiction.

We show this by pulling out the forcing $S^1_{\omega \alpha + 2n+1}$ out of the forcing $\forceP^0 \times \forceP^1 \ast \forceP$ over $L$ which produces $W[G]$. 
Indeed if we consider $W[\forceP]=L[\forceP^0] [\forceP^1][\forceP]$, and if $S^1_{\omega \alpha + 2n+1}$ is as described already, by the reasoning above
we can rearrange the generics to $W[G]= L  [G'^0 \times G^0_{\omega \alpha + 2n+1} ] [G^1] [G] = L[G'] [G^0_{\omega \alpha + 2n+1} ]$, where $G'^0$ is $\prod_{\beta \ne \alpha+2n+1}  G^0_{\beta}$ and $G'$ is $G'^0 \times G^1 \ast G$.

Note now that $S_{\alpha+2n+1}$ is still a Suslin tree in $L[G']$ so the forcing $S^1_{\alpha+2n+1}$ is $\omega$-distributive. This can be seen using the fact that $\vec{S}^0$ and $\vec{S}^1$ are independent. Indeed $S^1_{\alpha+2n+1}$ will remain Suslin in $L[G'^0][G^1]$, as we can write the universe as $L[G^1] [G'^0]$. But then the finite support iteration of the almost disjoint coding forcings is Knaster, hence keeps $S^1_{\alpha+2n+1}$ Suslin. Consequently $2^{\omega} \cap W[G] = 2^{\omega} \cap L[G']$. 

 But this implies that 
\[L[G'] \models \lnot \exists r L[r] \models `` S^1_{\alpha+2n+1} \text{ has an $\omega_1$-branch}" \]
as the existence of an $\omega_1$-branch through $S^1_{\alpha+2n+1}$ in the inner model $L[r]$ would imply the existence of such a branch in $L[G']$. Further, as no new reals appear when passing to $W[G]$ we also get 
\[W[G] \models \lnot \exists r L[r] \models `` S^1_{\alpha+2n+1} \text{ has an $\omega_1$-branch}". \]
This is the desired contradiction.

\end{proof}

\subsection{$\alpha$-allowable forcings}

The notion of 0-allowable will form the base case of an inductive definition.
Let $\alpha \ge 0$ be an ordinal and assume we defined already the notion of $\alpha$-allowable relative to a bookkeeping $F$ and a set $E$. For $0$-allowable we define the set $E$ to be empty. Then we can inductively define the notion of $\alpha+1$-allowable with respect to $E$ and $F$  as follows.

We list all  ${\Sigma}^1_3$-formulas $\varphi_n (v_0,v_1)$ with two free variables. If $y$ is a real then $\varphi_n(v_0,y)$ defines a $\Sigma^1_3(y)$-set $A_n(y_0)$, where $x \in A_n(y) \leftrightarrow \varphi_n(x,y)$.
Suppose that $\gamma < \omega_1$, $F$ is a bookkeeping function,  \[F: \gamma \rightarrow H(\omega_2) \] and \[\forceP=(\forceP_{\beta} \, : \, \beta < \gamma)\] is a allowable forcing relative to $F$.

Suppose that 
\[E= E_0 \cup E_1 \] 
where\[ E_0:= \{(\dot{y}_{\delta}, \dot{m}_{\delta} ,\dot{k}_{\delta}) \, : \,  \delta \le \alpha\} \]
and 
\[E_1 := \{(\dot{x}_{\delta},i_{\delta} ) \, : \,\delta \le \alpha, i_{\delta} \in \{0,1\} \} \]
where we allow $(\dot{y}_{\delta}, \dot{m}_{\delta} ,\dot{k}_{\delta})$ and $(\dot{x}_{\delta},i_{\delta} )$ to be the empty set. Otherwise  $\dot{m}_{\delta},\dot{k}_{\delta} $ are names for elements of $ \omega$ and $\dot{x}_{\delta}, \dot{y}_{\delta}$ are $\forceP$-names of two reals and for every two ordinals $\beta < \gamma \le  \alpha$,  if $(\dot{y}_{\beta}, \dot{m}_{\delta},\dot{k}_{\delta})$ and $(\dot{y}_{\gamma},\dot{m}_{\gamma},\dot{k}_{\gamma})$ are not the empty set, then $\forceP \Vdash (\dot{y}_{\beta}, \dot{m}_{\beta},\dot{k}_{\beta}) \ne (\dot{y}_{\gamma},\dot{m}_{\gamma},\dot{k}_{\gamma})$. Intuitively, $E_0$ will serve as the set of pairs of boldface $\Sigma^1_3$-sets, for which we already obtained rules which allow us to separate them; whereas $E_1$ is the set of (names of) reals which we decided to never code along our $\alpha$-allowable iteration using the coding forcing $\operatorname{Code}(x,a,b,i,\eta)$ for two fixed natural numbers $a$ and $b$ and any ordinal $\eta < \omega_1$. The latter plays a key role in establishing the eventual failure of $\Pi^1_3$-reduction.

Suppose that for every $\delta \le \alpha$, $(\forceP_{\beta} \, : \, \beta < \gamma )$ is $\delta$-allowable with respect to $E \upharpoonright \delta =  (E_0 \upharpoonright \delta) \cup (E_1 \upharpoonright \delta ) =\{ ( \dot{y}_{\eta},\dot{m}_{\eta},k_{\eta})  \, : \, \eta < \delta \} \cup \{ (\dot{x}_{\eta}, i_{\eta} ) \, : \, \eta < \delta\}$ and $F$. Our definition will consider two cases. In the first case we shall define $\alpha+1$-allowable with respect to $E_0 \cup \{\dot{y}_{\alpha+1}, \dot{m}_{\alpha+1},\dot{k}_{\alpha+1} ) \}$, that is we enlarge $E_0$ while keeping the old $E_1$ (via letting $(\dot{x}_{\alpha+1}, i_{\alpha+1})$ be the empty set) and say what $\alpha+1$-allowable should mean. In the second case we enlarge $E_1$, while keeping the old $E_0$ and say what $\alpha+1$-allowable should mean. We assume as a part of the definition of $\alpha$-allowable that we did not enlarge both $E_0$ and $E_1$ when passing from $\beta$-allowable to $\beta+1$-allowable for any $\beta < \alpha$. This property will be kept when defining $\alpha+1$-allowable which is what we do now.

\subsubsection{$\alpha+1$-allowable forcings}

Suppose $\mathbb{P} = (\mathbb{P}_{\beta} : \beta < \gamma)$ is an $\alpha$-allowable forcing with respect to $E = E_0 \cup E_1$ and a bookkeeping function $F$. Since $\mathbb{P}$ is allowable, it inherently satisfies the Freshness Condition at every stage (meaning $\mathbb{P}(\beta)^{G_{\beta}}$ is defined to be the trivial forcing whenever the Freshness Condition fails). 

We define $\mathbb{P}$ to be $\alpha+1$-allowable by specifying two mutually exclusive ways to extend $E$ to $E'$ and imposing additional rules on the iterands at stages where the Freshness Condition holds.

\paragraph{Case 1: Enlarging $E_0$}
Assume $(\dot{x}_{\alpha+1}, i_{\alpha+1})$ is empty, $\dot{y}_{\alpha+1}$ is a $\mathbb{P}$-name for a real, and $m_{\alpha+1}, k_{\alpha+1} \in \omega$ such that $\mathbb{P} \Vdash \forall \delta \le \alpha \, ((\dot{y}_{\delta}, m_{\delta}, k_{\delta}) \ne (\dot{y}_{\alpha+1}, m_{\alpha+1}, k_{\alpha+1}))$. We define $\mathbb{P}$ to be $\alpha+1$-allowable with respect to $F$ and $E' = E_0 \cup \{(\dot{y}_{\alpha+1}, m_{\alpha+1}, k_{\alpha+1})\} \cup E_1$ if, at every stage $\beta < \gamma$ where the Freshness Condition holds, the iterand $\mathbb{P}(\beta)^{G_{\beta}}$ obeys the following rules. Let $G_{\beta}$ be $\mathbb{P}_{\beta}$-generic over $W$, and evaluate $F(\beta)$ using $G_{\beta}$:

\begin{enumerate}
    \item \textbf{Separation towards the $A_m$-side:} If $F(\beta) = (\dot{x}, \dot{y}_{\alpha+1}, m_{\alpha+1}, k_{\alpha+1}, i, \dot{\eta})$ and $W[G_{\beta}]$ satisfies:
    $$ \exists \mathbb{Q} \, (\mathbb{Q} \text{ is } \alpha \text{-allowable w.r.t } E \text{ and some } F' \land \mathbb{Q} \Vdash x \in A_m(y_{\alpha+1})) $$
    Then the chosen forcing must be:
    $$ \mathbb{P}(\beta)^{G_{\beta}} := \operatorname{Code}(x, y_{\alpha+1}, m_{\alpha+1}, k_{\alpha+1}, 0, \eta) $$

    \item \textbf{Separation towards the $A_k$-side:} If Rule 1 fails, $F(\beta) = (\dot{x}, \dot{y}_{\alpha+1}, m_{\alpha+1}, k_{\alpha+1}, i, \dot{\eta})$, and $W[G_{\beta}]$ satisfies:
    $$ \exists \mathbb{Q}_2 \, (\mathbb{Q}_2 \text{ is } \alpha\text{-allowable w.r.t } E \text{ and some } F' \land \mathbb{Q}_2 \Vdash x \in A_k(y_{\alpha+1})) $$
    Then the chosen forcing must be:
    $$ \mathbb{P}(\beta)^{G_{\beta}} := \operatorname{Code}(x, y_{\alpha+1}, m_{\alpha+1}, k_{\alpha+1}, 1, \eta) $$

    \item \textbf{Separation following bookkeeping:} If neither Rule 1 nor Rule 2 applies for $F(\beta) = (\dot{x}, \dot{y}_{\alpha+1}, m_{\alpha+1}, k_{\alpha+1}, i, \dot{\eta})$, then:
    $$ \mathbb{P}(\beta)^{G_{\beta}} := \operatorname{Code}(x, y_{\alpha+1}, m_{\alpha+1}, k_{\alpha+1}, i^{G_{\beta}}, \eta) $$

    \item \textbf{Default Coding:} If $F(\beta) = (\dot{x}, \dot{y}, \dot{m}, \dot{k}, i, \dot{\eta})$ evaluates to a tuple such that:
    $$ W[G_{\beta}] \models ((y, m, k) \notin E^{G_{\beta}} \cup \{(y_{\alpha+1}, m_{\alpha+1}, k_{\alpha+1})\}) $$
    where $E^{G_{\beta}}$ is the set of evaluations of names in $E$ with the help of $G_{\beta}$. Then:
    $$ \mathbb{P}(\beta)^{G_{\beta}} := \operatorname{Code}(x, y, m, k, i^{G_{\beta}}, \eta) $$
\end{enumerate}

\paragraph{Case 2: Enlarging $E_1$}
Assume $(\dot{y}_{\alpha+1}, \dot{m}_{\alpha+1}, \dot{k}_{\alpha+1})$ is empty, $\dot{x}_{\alpha+1}$ is a $\mathbb{P}$-name of a real, and $i_{\alpha+1}$ is a $\mathbb{P}$-name of an element in $\{0,1\}$. 

We define $\mathbb{P}$ to be $\alpha+1$-allowable with respect to $F$ and $E' = E_0 \cup E_1 \cup \{(\dot{x}_{\alpha+1}, i_{\alpha+1})\}$ if it is $\alpha$-allowable relative to $E$ and $F$, subject to the following additional constraint for our iteration $\mathbb{P}=(\mathbb{P}_{\beta} \mid \beta < \gamma)$:
\begin{align*}
    \forall \beta < \gamma \Big( (\dot{x}_{\alpha+1}, i_{\alpha+1} &\text{ are } \mathbb{P}_{\beta}\text{-names}) \implies \\
    & \forall \eta < \omega_1 \big( \mathbb{P}(\beta)^{G_{\beta}} \neq \operatorname{Code}(\dot{x}_{\alpha+1}^{G_{\beta}}, a, b, i_{\alpha+1}^{G_{\beta}}, \eta) \big) \Big).
\end{align*}
(Here $a, b$ represent the G\"odel numbers of our two distinguished $\Sigma^1_3$ formulas).
To establish some jargon, the above defined constraint is often referred to as \emph{not using the tuple} $(x_{\alpha+1},a,b,i)$ \emph{for coding.} We also say in this situation that the coding forcing
$\operatorname{Code}(\dot{x}_{\alpha+1}^{G_{\beta}}, a, b, i_{\alpha+1}^{G_{\beta}}, \eta)$ is never used as a factor in an $\alpha+1$-allowable forcing.

We also allow enlarging $E_1$ by any infinite set of pairs (of names for reals and indicators in $\{0,1\}$), with the demand that we never use the associated coding forcings when defining $\alpha+1$-allowable iterations with respect to $E'$ and $F$.

\subsubsection{Limit stages}

For limit ordinals $\alpha$, we say that an allowable forcing $\forceP$ is $\alpha$ allowable with respect to $E$ and $F$ if for every $\eta < \alpha$,
$(\forceP_{\beta} \,: \, \beta < \gamma)$ is $\eta$-allowable with respect to $E \upharpoonright \eta$ and some $F'$.
\par \medskip

\subsubsection{Remarks}
We add a couple of remarks concerning the last definition.
\begin{itemize}

\item By definition, if $\delta_2 < \delta_1$ and $\forceP_1$ is $\delta_1$-allowable with respect to $E= \{(\dot{y}_{\beta}, \dot{m}_{\beta} ,\dot{k}_{\beta}) \, : \,  \beta \le \delta_1\}\cup \{(\dot{x}_{\beta}, i_{\beta} ) \, : \, \beta \le \delta_1$ and some $F_1$, then  $\forceP_1$ is also $\delta_2$-allowable with respect to $E \upharpoonright \delta_2 = \{( \dot{y}_{\beta}, \dot{m}_{\beta} ,\dot{k}_{\beta}) \, : \,  \beta \le \delta_2\} \cup \{ ( \dot{x}_{\beta}, i_{\beta} ) \, : \, \beta \le \delta_2, i_{\beta} \in \{0,1\} \}$ and an altered bookkeeping function $F'$. The bookkeeping $F'$ should just use the relevant stages where case 1.3 applies to guess what the $\delta_1$-allowable forcing $\forceP_1$ with respect to $E$ and $F$ does at these stages when in cases 1.1 or 1.2.
 
\item The notion of $\alpha$-allowable can be defined in a uniform way over any allowable extension $W'$ of $W$, analogous to the case $\alpha=0$. 

\item We will often just say that some iteration $\forceP$ is $\alpha$-allowable, by which we mean that there is a set $E$ and a bookkeeping $F$ such that $\forceP $ is $\alpha$-allowable with respect to $E$ and $F$.
\end{itemize}

The two step iteration of $\alpha$-allowable forcings will result in an $\alpha$-allowable forcing again, provided we define the right notion of what it means for a forcing to be $\alpha$-allowable over an $\alpha$-allowable generic extension of $W$. The next definition is by induction on $\alpha$. We assume that we do have defined already  the notion of $\beta$-allowable over $\beta$-allowable extensions of $W$ for every $\beta < \alpha$. Then we define $\alpha$-allowable over an $\alpha$-allowable generic extension of $W$ as follows.

\begin{definition}\label{def:alpha_allowable_extension}
Let $\mathbb{P}$ be an $\alpha$-allowable forcing over $W$ of length $\delta_1$ with respect to $E$ and $F_{\mathbb{P}}$, and let $G_{\mathbb{P}}$ be a $\mathbb{P}$-generic filter over $W$. A finite support iteration $\mathbb{Q} = (\mathbb{Q}_{\beta} : \beta < \delta_2)$ in $W[G_{\mathbb{P}}]$ is \emph{$\alpha$-allowable over the extension $W[G_{\mathbb{P}}]$} with respect to $E$ and a bookkeeping function $F_{\mathbb{Q}} : \delta_2 \to L$ if it satisfies the inductive definition of $\alpha$-allowability with the following changes:

\begin{enumerate}
    \item The ground model $W$ is replaced by $W[G_{\mathbb{P}}]$. Consequently, at any stage $\beta < \delta_2$, the model evaluations guiding the iteration (e.g., $W[G_{\mathbb{P}}][G_{\mathbb{Q}_\beta}] \models \dots$) are performed over the intermediate extension of $W[G_{\mathbb{P}}]$ by $G_{\mathbb{Q}_\beta}$.
    \item The base $0$-allowability requirements are subject to the Global Freshness Condition. For any coding area $\eta$ proposed at stage $\beta$ of $\mathbb{Q}$, we require $\eta \notin C^{G_{\mathbb{P}}} \cup C^{G_{\mathbb{Q}_\beta}}$. If this fails, $\mathbb{Q}(\beta)^{G_{\mathbb{Q}_\beta}}$ is defined as the trivial forcing.
    \item  For any $\alpha$, whenever the rules for $\alpha$-allowability require the existence of a $<\alpha$-allowable forcing $\mathbb{S}$ to dictate separation (e.g., $\mathbb{S} \Vdash x \in A_m(y)$), this witnessing forcing $\mathbb{S}$ must be $<\alpha$-allowable \emph{over} $W[G_{\mathbb{P}}][G_{\mathbb{Q}_\beta}]$.
\end{enumerate}
\end{definition}

The following lemma's proof is a straightforward application of the definition (see the almost identical proof for Lemma \ref{iteration_of_allowable}).
\begin{lemma}
   Fix an ordinal $\alpha$. Let $\forceP$ be an $\alpha$-allowable forcing of length $\delta_1$ and suppose $\forceP \Vdash ``\dot{\forceQ}$ is $\alpha$-allowable over $W[\dot{G}_{\forceP}]$ of length $\delta_2$''. Then the two step iteration $\forceP \ast \dot{\forceQ}$ is an allowable forcing over $W$ of length $\delta_1 + \delta_2$.
\end{lemma}

\section{Closure under products}

\begin{lemma}\label{productallowable}
Let $\alpha$ be an ordinal, assume that $W'$ is some $\alpha$-allowable generic extension of $W$, and that $\forceP^1=(\forceP^1_{\beta} \,: \,\beta < \delta_1)$ and $\forceP^2=(\forceP^2_{\beta} \,: \, \beta < \delta_2) $ are two $\alpha$-allowable forcings over $W'$ with respect to a common set $E=E_0 \cup E_1=\{(\dot{y}_{\delta},\dot{m}_{\delta}, \dot{k}_{\delta}) \,: \, \delta < \alpha\} \cup \{ (\dot{x}_{\delta},i_{\delta}) \, : \, \delta < \alpha \}$  and bookkeeping functions $F_1$ and $F_2$ respectively. Further assume that 
\[ C^{\forceP^1} \cap C^{\forceP^2} = \emptyset, \]
and that $E_1= \{ (\dot{x_{\delta}},i_{\delta}) \mid \delta < \alpha \}$ consists of reals which are in fact elements of $W'$.
Then there is a bookkeeping function $F$ such that $\forceP^1 \times \forceP^2$ densely embeds into an $\alpha$-allowable forcing over $W'$ with respect to $E$ and $F$.
\end{lemma}
\begin{proof}
We shall show that $\forceP^1 \ast \check{\forceP}^2$ is $\alpha$-allowable with respect to $E$ and some $F$. This iteration should then serve as the asserted $\alpha$-allowable forcing in which $\forceP^1 \times \forceP^2$ densely embeds. 
We define $F \upharpoonright \delta_1$ to be $F_1$. For values $\delta_1+ \beta  \ge \delta_1$ we let $F(\delta_1+\beta)$ be $F_2(\beta)$ where we identify $\forceP^2_{\beta}$-names with the $\forceP^1$-check names for $\forceP^2_{\beta}$-names. The lemma is proved by induction on $\alpha$. If $\alpha=0$, then this is Lemma \ref{FirstPropertiesOfAllowableForcings}. 

So we assume that $\alpha=1$. We assume first that, we have enlarged $E_1$ in order to get the notion of 1-allowable. Then there is a real $x_0 \in W'$ and $i \in \{0,1\}$ such that 1-allowable with respect to $E$ and some $F$ is just an allowable forcing with the additional constraint, that we must not use forcings of the form $\operatorname{Code} (x_0,y,a,b,i, \eta)$ for every $\eta$. This means that
    \begin{align*}
        \forall \eta < \delta_1 ( \forall \zeta < \omega_1 (( \mathbb{P}^1(\eta)^{G^1_{\eta}} \neq \operatorname{Code}(x_0, a, b, i, \zeta) )).
    \end{align*}
    and
    \begin{align*}
        \forall \eta < \delta_2 ( \forall \zeta < \omega_1 (( \mathbb{P}^2(\eta)^{G^2_{\eta}} \neq \operatorname{Code}(x_0, a, b, i, \zeta) )).
    \end{align*}
    This property is clearly preserved under $\forceP^1 \times \forceP^2$.
    So if both $\forceP^1$ and $\forceP^2_{\beta}$ are 1-allowable with respect to $E $ and $F_1$ and $F_2$ respectively, then, at stage $\beta$ $\forceP^1 \times \forceP^2$ is 1-allowable with respect to $F$ and $E$ as asserted.

So we can assume that 1-allowable was obtained via enlarging $E=E_0= \{(\dot{y},\dot{m},\dot{k})\}$. 
We assume first that $W= W'$. It will become clear once the proof is finished for this case that extending the proof to a $W'$ is almost trivial.
Let $(\forceP_2)_{\beta}$ be the iteration of $\forceP_2$ up to stage $\beta < \delta_2$. Assume, that $\forceP_1 \times (\forceP_2)_{\beta}$ is in fact an $\alpha$-allowable forcing relative to $E$ and $F$. Let $(G^1 \times G^2_{\beta})$ be a $\forceP^1 \times \forceP^2_{\beta}$-generic filter over $W$.  Then we have that $F(\delta_1+\beta) =F_2(\beta)  =(\dot{x},\dot{y},\dot{m},\dot{k} ,i,\dot{\eta})$, and we claim that

\begin{claim}
At any stage $\beta < \delta_2$, if we evaluate the names in $F_2(\beta)$ using $G^2_\beta$, the specific case (Case 1.1, 1.2, 1.3, 1.4 from the definition of $\alpha$-allowability) that applies in the model $W[G^2_\beta]$ is exactly the same case that applies in the extended model $W[G^1][G^2_\beta]$.
\end{claim}

If the claim holds, the lemma follows immediately by induction on $\beta < \delta_2$. Working in $W[G^1][G^2_\beta]$, we evaluate $F(\delta_1 + \beta) = F_2(\beta)$. By the claim, the case determining $\mathbb{P}^2(\beta)^{G^2_\beta}$ over the larger model $W[G^1][G^2_\beta]$ is identical to the one over $W[G^2_\beta]$, ensuring the rules of $\alpha$-allowability are strictly obeyed.

\medskip

\begin{comment}
    To prove the claim  we need to show that  $(\forceP^2_{\beta}\, : \, \beta < \delta_2)$ is still 1-allowable with respect to $E$ and $F$, when considered in the universe $W[\forceP^1]$. We fix a $\prod_{i < \omega_1} S_{i} \times \prod_{i < \omega_1} \mathbb{C} (\omega_1)$-generic filter $G \times H$ over $L$, where $G$ should be generic for the Suslin tree forcing and $H$ is generic for the $\omega_1$-Cohen forcing. Then $W$ is assumed to be a 0-allowable extension of $L[G\times H]$, and we let $\forceP$ be such that $L[G \times H] [\forceP]=W$. We note that the missing $\aleph_1$-many $h_j$'s from $H$ which are not used in the definition of $\forceP$ will not add reals to $L[G] [\{C_i \mid i \in \alpha < \omega_1 \}] [\forceP]$ by Easton's Lemma applied over $L[\{C_i \mid i \in \alpha < \omega_1 \}]$, that is 
\[ (P(\omega))^{L[G][\{C_i \mid i \in \alpha < \omega_1 \}] [\forceP]}= (P(\omega))^{L[G][H][\forceP]}.\]
So $\forceP$ can correctly be defined in an inner model of $L[G \times H]$ which only uses countably many coordinates of $H$, so $\forceP \in L[G] [ \{C_i \mid i \in \alpha < \omega_1 \}]$.
With this is mind we can assume without loss of generality that $W=W$. The proof will carry over to the case $W\ne W$ in a straightforward way.
\end{comment}

We assume first that at stage $\beta$, in the definition of $ \forceP^2$,  case 1.1 in the definition of 1-allowable applies, when working in the model $W[G^2_{\beta}]$ relative to $E$ and $F_2$. So we assume that $\beta$ is a stage such that
\[ F_2(\beta) = (\dot{x},\dot{y},\dot{m},\dot{k},i,\dot{\eta}) \]
and $(\dot{y},\dot{m},\dot{k}) \in E$ and if $\dot{x}^{G^2_{\beta}}=x$ and $\dot{y}^{G^2_{\beta}}=y$, the universe $W[ G^2_{\beta}]$ thinks that
\begin{align*}
\exists \forceQ (&\forceQ \text{ is } 0 \text{-allowable with respect to } E \text{ and some F' }\,  \land \\& \, \forceQ \Vdash {x} \in A_m({y})).
\end{align*}
Now when working over $W[G^1 \times G^2_{\beta}]$ instead, we shall show that
we are still in case 1.1.

The idea is that $ \forceQ$, as it will force $x \in A_m(y)$, will witness that we are in case 1 as well when working over $W[G^1][G^2_{\beta}]$. For that we need to show that $\forceP^1 \times \forceQ$ is 0-allowable, or in other words, that their coding areas do not intersect. As there is no reason why this should be the case, we need an additional argument. For that we show that the existence of a 0-allowable forcing which forces some $\Sigma^1_3$-statement is not tied to the coding areas it uses, as they are generically added sets. This will imply that we can shift the coding areas of $\forceQ$ away from $C^{G^1}$ and still obtain allowable forcings using these shifted coding areas instead while still forcing the $\Sigma^1_3$-assertion to hold.

We shall explain this shifting mechanism now. We write $W= L[G]$ and let $(G)_0$ and $(G)_1$ denote its projection to the first and the second coordinate respectively, so that $(G)_0$ is $\forceP^0= \prod_{i < \omega_1} \vec{S}$-generic and $(G)_1$ is $\forceP^1= \prod_{i < \omega_1} \mathbb{C}(\omega_1)^L$-generic. We write $(G)_1$ as $(C_i \mid i < \omega_1)$.  As $ G^2_{\beta}$ picks at each stage exactly one coding area $C_i$ there is $\eta < \omega_1$ and we can write $L[(G)_0][\{C_i \mid i < \omega_1 \}][G^2_{\beta}]$ as $L[(G)_0][\{C_i \mid i < \eta \}] [G^2_{\beta}] [\{C_i \mid \eta \le i < \omega_1 \}]$ and  the coding areas of $\forceP^2_{\beta}$ relative to $G^2_{\beta}$ is a subset of $\eta$.
\begin{align*}
L[(G)_0] [\{C_i \mid i < \eta \}] [G^2_{\beta}] [\{C_{\eta+ j} \mid j < \omega_1\} ] \models  \exists \forceQ (&\forceQ \text{ is } 0 \text{-allowable with} \\&
\text{respect to } E \text{ and some F' }  \\& \land \forceQ \Vdash {x} \in A_m({y})).
\end{align*}
Thus there is a condition $p \in \prod_{j<\omega_1} C_{\eta+j} \subset (\prod_{\eta \le j <\omega_1} \mathbb{C} (\omega_1))^L$
such that 
\begin{align*}
L[G] [\{C_i \mid i < \eta \}] [G^2_{\beta}] \models p  \Vdash_{(\prod_{j \in [\eta,\omega_1)} \mathbb{C} (\omega_1))^L}  \exists \forceQ (&\forceQ \text{ is } 0 \text{-allowable with} \\&
\text{respect to } E \text{ and some F' }  \\& \land \forceQ \Vdash {x} \in A_m({y})).
\end{align*}
As $(\prod_{j \in [\eta,\omega_1)} \mathbb{C} (\omega_1))^L$ is isomorphic to any 
$(\prod_{ i \in [\zeta,\omega_1)} \mathbb{C} (\omega_1))^L$, for $\zeta < \omega_1$, we can find a sufficiently large $\zeta< \omega _1$ such that the interval $[\zeta, \omega_1)$ is fully disjoint from the coding areas from $\forceP^1$ relative to $G^1$. We can also find a condition $p' \in (\prod_{j \in [\zeta,\omega_1)} \mathbb{C} (\omega_1))^L$ such that 
\begin{align*}
L[G] [\{C_i \mid i < \eta \}] [G^2_{\beta}] \models p'  \Vdash_{(\prod_{j \in [\zeta,\omega_1)} \mathbb{C} (\omega_1))^L}  \exists \forceQ (&\forceQ \text{ is } 0 \text{-allowable with} \\&
\text{respect to } E \text{ and some $F'$ }  \\& \land \forceQ \Vdash {x} \in A_m({y})).
\end{align*}
Note that $(\prod_{i< \omega_1} \mathbb{C} (\omega_1))^L$ is weakly homogeneous \footnote{The homogeneity of the forcing is in fact not necessarily needed for this argument and its use can be replaced by a density argument. We use the homogeneity as the argument becomes neater this way.}, so in fact
\begin{align*}
L[G] [\{C_i \mid i < \eta \}] [G^2_{\beta}] \models 1  \Vdash_{(\prod_{j \in [\zeta,\omega_1)} \mathbb{C} (\omega_1))^L}  \exists \forceQ (&\forceQ \text{ is } 0 \text{-allowable with} \\&
\text{respect to } E \text{ and some $F'$}  \\& \land \forceQ \Vdash {x} \in A_m({y})).
\end{align*}
So for any $\zeta < \omega_1$, $\zeta > \eta$ we can find a condition $q \in (\prod_{j \in [\zeta,\omega_1) } \mathbb{C} (\omega_1))^L$ such that $q \in \prod_{j \in [\zeta,\omega_1) } C_j$ and
\begin{align*}
L[G] [\{C_i \mid i < \eta \}] [G^2_{\beta}] \models q  \Vdash_{(\prod_{j \in [\zeta,\omega_1)} \mathbb{C} (\omega_1))^L}  \exists \forceQ (&\forceQ \text{ is } 0 \text{-allowable with} \\&
\text{respect to } E \text{ and some $F'$ }  \\& \land \forceQ \Vdash {x} \in A_m({y})).
\end{align*}
As the coding areas of $\forceP^1$ are countable, we can find a $\zeta> \eta$ such that $C^{\forceP^1} \cap [\zeta, \omega_1) = \emptyset$. Also the coding areas of $\forceQ$ are contained as a subset in the interval $[\zeta,\omega_1)$ so we can conclude
\begin{align*}
L[(G)_0] [(G)_1] [G^2_{\beta}] \models  \exists \forceQ (&\forceQ \text{ is } 0 \text{-allowable with} \\&
\text{respect to } E \text{ and some $F'$ }  \\&  \land C^{\forceQ} \cap C^{\forceP^1}=\emptyset \\& \land
\forceQ \Vdash {x} \in A_m({y})).
\end{align*}
and so
\begin{align*}
L[(G)_0] [(G)_1] [G^1] [G^2_{\beta}] \models  \exists \forceQ (&\forceQ \text{ is } 0 \text{-allowable with} \\&
\text{respect to } E \text{ and some $F'$ }  \\&  \land
\forceQ \Vdash {x} \in A_m({y})).
\end{align*}

So at stage $\beta$, we are in case 1 as well, when working over $W[G^1][G^2_{\beta}]$ which proves the Claim in the first instance.

If, at stage $\beta$ and working in the model $W[G_{\beta}^2]$ case 2 applies, when considering $ \forceP_2$ as a 1-allowable forcing with respect to $E$ and $F^2$ over $W$, then we argue first that case 1 is impossible when considering $\forceP_2$ as a $1$-allowable forcing over $W[G^1]$, where as $G^1$ is $\forceP^1$-generic as before.

Indeed, assume for a contradiction that case 1 must be applied at stage $\beta$, then, by definition,
\begin{align*}
W [G^1] [G^2_{\beta}] \models  \exists \forceQ (&\forceQ \text{ is } 0 \text{-allowable with} \\&
\text{respect to } E \text{ and some $F'$}  \\&  \land
\forceQ \Vdash {x} \in A_m({y})).
\end{align*}
But then $\forceP^1 \times \forceQ$ is such that there is a condition $p^1 \in \forceP^1 \cap G^1$ which forces that $p^1 \Vdash C^{\forceP^1} \cap C^{\forceQ} = \emptyset$. So restricting $\forceP^1 \times \forceQ$ to conditions stronger than $(p^1,1)$ will yield a $0$-allowable forcing $(\forceP^1 \times \forceQ)_{\le (p^1,1)} $ such that 
\begin{align*}
W [G^2_{\beta}] \models  (&(\forceP^1 \times \forceQ)_{\le (p^1,1)} \text{ is } 0 \text{-allowable with} \\&
\text{respect to } E \text{ and some $F'$ }  \\&  \land
(\forceP^1 \times \forceQ)_{\le (p^1,1)} \Vdash {x} \in A_m({y})).
\end{align*}
So we are in case 1 as well when working over $W [G^2_{\beta}]$, which is a contradiction to our assumption.

Now we need to argue that we are in case 2, when working over $W[G^1][G^2_{\beta}]$ in the definition of $\forceP^2$. We can argue in a very similar way than before. So suppose we are at stage $\beta$ of our iteration and we work over the model $W[G^1][G^2_{\beta}]$ and case 1 does not apply, which we know already. The universe $W[G^2_{\beta}]$ thinks that there is $\forceQ$ which is 0-allowable and which forces $x \in A_k(y)$. Just as before we would like to form $\forceP^1 \times \forceQ$ in the model $W[G^2_{\beta}]$ which will force $x \in A_k(y)$, therefore witnessing that we are in case 2, which is what we want to finish the proof. Again, as before, there is no guarantee that $\forceP^1$ and $\forceQ$ have non-intersecting coding areas, so we need to apply again the shifting argument from above.
As there we can argue that there is in fact an allowable $\forceQ$ such that $C^{\forceQ} \cap C^{\forceP^1}= \emptyset$ and such that $\forceQ \Vdash x \in A_k (y)$. Now this $\forceQ$ can be taken to form a product with $\forceP^1$, which witnesses that we are in case 2 when defining $\forceP^2$ at stage $\beta$ over the model $W[G^2_{\beta}]$ so this proves the claim for case 2 in the definition of allowable forcings.

Finally, if at stage $\beta$, case 3 applies when considering $\forceP_2$ as a 1-allowable forcing with respect to $E$ over $W[G^2_{\beta}]$, we claim that we must be in case 3 as well, when defining $\forceP_2 (\beta)$ over $W[G^1][G^2_{\beta}]$.

If not, then we would be in case 1 or 2 when defining $\forceP^2(
\beta)$ over $W[G^1][G^2_{\beta}]$. Assume without loss of generality that we were in case 1, then there would be a 0-allowable $\forceQ \in W[G^1][G^2_{\beta}]$ such that $\forceQ \Vdash x \in A_m(y)$. Again, via shifting the coding areas of $\forceQ$ we can assume that the forcing $\forceP^1 \ast\forceQ $ in $W[G^2_{\beta}]$ is allowable and witnesses that we are in fact in case 1 as well when defining  $\forceP^2(\beta)$ over $W[G^2_{\beta}]$ which is a contradiction to our assumption that we are in case 3. This shows that we must be in case 3 as well, and the claim, and hence the lemma is finally proved for $\alpha=1$.

\medskip

We shall argue now that the Claim is true for $\alpha+1$-allowable forcings provided we know that it is true for $\alpha$-allowable forcings. Again we can focus on the case when $\alpha+1$ allowable forcings are obtained via enlarging $E_0$, as enlarging $E_1$ just means to avoid certain coding forcings, which is trivial to be closed under products.
We shall show the claim via induction on $\beta$. The proof itself is almost identical to the case $\alpha=1$. As there, we can shift around coding areas using isomorphisms of $(\prod_{\omega_1} \mathbb{C} (\omega_1))^L$ to always find the relevant $\alpha$-allowable forcings whose coding areas are disjoint. All the work we have to do is replace $0$-allowable with $\alpha$-allowable.

This ends the proof of the claim and so we have shown the lemma.
\end{proof}
\section{Ideas for the proof}

This section will be used to briefly explain the set up and the structure of the proof of the main theorem.
There are two goals we aim to accomplish. First, we want to force $\bf{\Sigma}^1_3$-separation. For this, we will use an $\omega_1$-length iteration. We list all $\Sigma^1_3$-formulas in two free variables $(\varphi_n (v_0,v_1) \mid n \in \omega)$. We shall use a bookkeeping device which  enumerates simultaneously  in an $\omega_1$-length list all pairs of natural numbers $(m,k) \in \omega^2$ and (names of) reals $\dot{y}$. These objects $m,k, \dot{y}$ correspond to pairs of $\bf{\Sigma}^1_3$-sets $A_m(\dot{y}) $ and $A_k (\dot{y})$ (the (name of a) real $\dot{y}$ serves as a parameter in the $k$-th and $m$-th $\Sigma^1_3$-formula $\varphi_m$ and $\varphi_k$) we want to separate.

At the same time we want to create a universe over which there are two (lightface) $\Pi^1_3$-sets $B_0$ and $B_1$ which we will design in such a way, that no (boldface) pair of $\bf{\Pi}^1_3$-sets exists, which reduces $B_0$ and $B_1$.

We settle to work towards $\bf{\Sigma}^1_3$-separation on the odd stages of our iteration, whereas we work towards a failure of $\Pi^1_3$-reduction on the even stages of the iteration.
The iteration itself will consist of the coding forcings $\operatorname{Code} (x,y,m,k,i,\eta)$ applied over $L$ to make certain reals of the form $(x,y,m,k)$ to satisfy our two $\Sigma^1_3$-formulas $\Phi_0 (x,y,m,k)$ or $\Phi_1(x,y,m,k)$. The final goal is that for any fixed pair of natural numbers $m,k$,  and any parameter $y \in \omega^{\omega}$, there is a real parameter $R_{y,m,k}$ and a fixed $\Sigma^1_3$ formula $\sigma$ such that the sets
\[ D^0_{y,m,k} (R_{y,m,k} ) := \{ x \mid  \Phi_0 (x,y,m,k) \land \sigma (x,R_{y,m,k}) \} \]
and 
\[ D^1_{y,m,k} (R_{y,m,k} ) := \{ x \mid \Phi_1 (x,y,m,k) \land \sigma (x,R_{y,m,k}) \} \]
will become the separating sets for the pair of $\Sigma^1_3 (y)$-definable sets $A_m(y)$ and $A_k(y)$, i.e. $A_m  (y) \subset D^0_{y,m,k}, A_k(y) \subset D^1_{y,m,k}$ and $D^0_{y,m,k} (R_{y,m,k} ) \cup D^1_{y,m,k} (R_{y,m,k} ) = \omega^{\omega}$ and 
$D^0_{y,m,k} (R_{y,m,k} ) \cap D^1_{y,m,k} (R_{y,m,k} ) = \emptyset$.

At the same time we need to work towards a failure of ${\Pi}^1_3$-reduction which will be done on the even stages of the iteration. We aim to accomplish the failure of $\Pi^1_3$-reduction via exhibiting two $\Pi^1_3$-sets (lightface) $B_0$ and $B_1$ which are chosen in such a way that the question of whether some real $x$ is in $B_0$ or $B_1$ can be changed using the coding forcings $\operatorname{Code}$ without interfering with the coding forcings we have to use in order to work for the $\bf{\Sigma}^1_3$-separation. This freedom will be used to define our $\omega_1$-length iteration of coding forcings so that eventually $B_0$ and $B_1$ can not be reduced by any pair of (boldface) $\bf{\Pi}^1_3$-sets, thus yielding a slightly stronger failure than just a failure of $\bf{\Pi}^1_3$-reduction.

\section{The first step of the iteration}
We let
\[\vec{\varphi} :=  (\varphi_n (v_0,v_1) \mid n \in \omega ) \]
be a fixed recursive list of the $\Sigma^1_3$-formulas in two free variables. We allow that $v_0$ or $v_1$ actually do not appear in some of the $\varphi_n$'s, so our list also contains all $\Sigma^1_3$-formulas in one free variable. We use $\psi_n $ to denote $\lnot \varphi_n$, i.e. $\psi_n$ is the $n$-th $\Pi^1_3$-formula in the recursive list of $\Pi^1_3$-formulas induced by $\vec{\varphi}$.

We fix two $\Sigma^1_3$-formulas $\varphi_{a}, \varphi_{b}$ which provably have non-empty intersection, e.g. $\varphi_a (v_0) = \exists v_0 (v_0 = v_0)$ and $\varphi_b= \exists v_0 (v_0 =v_0 \land v_0 =1)$. As a consequence, we need not to work towards separating $A_a$ and $A_b$, thus coding forcings of the form
$\operatorname{Code}(\dot{x},a,b,i,\eta)$ for any (name of a) real $\dot{x}$, for $a,b$ our fixed natural numbers and for $i \in \{0,1\}$ can be used freely in our definition of the iteration to come.

Next we assume for notational simplicity that $\lnot \varphi_0$ and $\lnot \varphi_1$, i.e. the negation of the first and the negation of the second formula in our list look like this:
\begin{align*}
&\lnot \varphi_0 (v_0) = \psi_0(v_0)= \lnot \Phi_0 (v_0) = ``(v_0,a,b) \text{ is not coded into the $\vec{S}^0$-sequence} "
\\& \lnot \varphi_1 (v_0)= \psi_1(v_0)= \lnot \Phi_1 (v_0)= ``(v_0,a,b) \text{ is not coded into the $\vec{S}^1$-sequence} "
\end{align*}
Recall 
\footnote{For a definition of $\Phi_i$ look right above Lemma \ref{nounwantedcodes}} that both $\Phi_0$ and $\Phi_1$ are $\Sigma^1_3$, hence $\psi)_i$ is $\Pi^1_3$. The resulting $\Pi^1_3$-sets will be defined like this:
\begin{align*}
&B_0 = \{ x \in \omega^{\omega} \mid \psi_0(x) \}
\\&  B_1= \{ x \in \omega^{\omega} \mid  \psi_1(x) \}
\end{align*}

Note that for any given real $x \in \omega^{\omega}$, we can manipulate
the truth value of $\psi_0(x)$ and $\psi_1(x)$ via using the coding forcings 
$\operatorname{Code} (x,a,b,0,\eta)$ and $\operatorname{Code} (x,a,b,1,\eta)$ respectively from true to false (and once false it will remain false because of upwards absoluteness of $\Sigma^1_3$-formulas). This in particular will not interfere with the yet to be defined procedure of forming the $\alpha_{\beta}$-allowable forcings, we will need in order to force $\bf{\Sigma}^1_3$-separation. 
Thus we gain some amount of flexibility in how we can make the sets $B_0$ and $B_1$ behave.  We will use this to diagonalise against all possible $\bf{\Pi}^1_3$-sets $B_m,B_k$ in such a way that none of those can reduce $B_0$ and $B_1$. This ensures the failure of $\Pi^1_3$-reduction.

\section{Towards $\bf{\Sigma}^1_3$-separation}

We are finally in the position to define the iteration which will yield a universe of $\Sigma_3^1$ separation and a failure of $\Pi_3^1$-reduction. The iteration we are about to define inductively will be an allowable iteration, whose tails are $\alpha$-allowable and $\alpha$ increases along the iteration.

We start by rigorously defining the bookkeeping function $F$. Since our ground model $W = L[G^0 \times G^1]$ is obtained via a ccc forcing over $L$, it inherits a canonical well-ordering $<_W$. For every $\xi < \omega_1$, the iteration $\mathbb{P}_{\xi}$ is a ccc forcing of size $\aleph_1$. Working in $W$, we can use $<_W$ to well-order the set of all nice $\mathbb{P}_{\xi}$-names for tuples of the form $(\dot{x}, \dot{y}, \dot{m}, \dot{k})$, where $\dot{x}, \dot{y}$ are names for reals, and $\dot{m}, \dot{k}$ are names for integers. Let this well-ordered list be enumerated as $\langle \tau_{\xi, \zeta} \mid \zeta < \omega_1 \rangle$.

We fix a surjective pairing function $f: \omega_1 \to \omega_1 \times \omega_1$ such that if $f(\beta) = (\xi, \zeta)$, then $\xi \le \beta$ (and strictly $\xi < \beta$ for $\beta > 0$). At stage $\beta < \omega_1$ of our iteration, we define the bookkeeping function $F(\beta)$ using this pairing:

Let $f(\beta) = (\xi, \zeta)$. We define $F(\beta)$ to be $\tau_{\xi, \zeta} = (\dot{x}, \dot{y}, \dot{m}, \dot{k})$, which is the $\zeta$-th tuple of $\mathbb{P}_{\xi}$-names from our well-ordered list. Because $\xi \le \beta$, the forcing $\mathbb{P}_{\xi}$ is a regular suborder of $\mathbb{P}_{\beta}$. Consequently, every $\mathbb{P}_{\xi}$-name canonically acts as a $\mathbb{P}_{\beta}$-name. 

When we evaluate the iteration at stage $\beta$ using a generic filter $G_{\beta}$, evaluating $F(\beta)$ with $G_{\beta}$ is equivalent to working in the intermediate model $W[G_{\xi}]$ (where $G_{\xi} = G_{\beta} \cap \mathbb{P}_{\xi}$) and looking at the evaluated reals $(x, y) = (\dot{x}^{G_{\xi}}, \dot{y}^{G_{\xi}})$ and integers $(m, k) = (\dot{m}^{G_{\xi}}, \dot{k}^{G_{\xi}})$. This  ensures that every possible tuple of real parameters and integer indices appearing in any intermediate extension $W[G_{\xi}]$ is handed to the iteration cofinally often. We can also  assume that every possible tuple of reals appears in the bookkeeping $F$ cofinally often on both the even and the odd ordinals below $\omega_1$.

Assume that we are at stage $\beta < \omega_1$ of our iteration.  By induction we will have constructed already the following list of objects.
\begin{itemize}

\item An ordinal $\alpha_{\beta} \le \beta$ and a set $E_{\alpha_{\beta}}= E^0_{\alpha_{\beta}} \cup E^1_{\alpha_{\beta}}$ which is of the form $\{(\dot{y}_{\eta}, \dot{m}_{\eta},\dot{k}_{\eta} ) \, : \, \eta < \alpha_{\beta} \} \cup \{ (\dot{x}_{\eta}, i_{\eta} ) \, : \, \eta < \alpha_{\beta} \}$, where $\dot{y}_{\eta},\dot{x}_{\eta} $ are $\forceP_{\beta}$-names of a reals, $\dot{m}_\eta, \dot{k}_{\eta}$ are names of natural numbers and $i_{\eta}$ a name of a set in $ \{0,1\}$. As a consequence, for every bookkeeping function $F'$ and every generic filter $G_{\beta}$ we do have a notion of $\eta$-allowable relative to $E$ and $F'$ over $W[G_{\beta}]$.

\item We assume by induction that for every $\eta < \alpha_{\beta}$, if $\beta_{\eta}< \beta$ is the $\eta$-th stage in $\forceP_{\beta}$, where we add a new member to $E_{\alpha_{\beta}}$, then $W[G_{\beta_{\eta}}]$ thinks that the $\forceP_{\beta_{\eta} \beta}$ is $\eta$-allowable with respect to $E_{\alpha_{\beta}} \upharpoonright \eta$.

\item If $( \dot{y}_{\eta},\dot{m}_{\eta},k_{\eta}) \in E_{\alpha_{\beta}}$, then we set again $\beta_{\eta}$ to be the $\eta$-th stage in $\forceP_{\beta}$ such that this new member to $E_{\alpha_{\beta}}$ is added. In the model $W[G_{\beta_{\eta}}]$, we can form the set of reals $R_{\eta}$ which were added so far by the use of a coding forcing in the iteration up to stage $\beta_{\eta}$, and  which witness $\Psi_i (x,y,m,k)$ holds for some $(x,y,m,k)$. 
So in $W[G_{\beta}]$, 
\begin{align*}
    R_{\eta}:= \{ R \mid & R \text{ is the generic filter for a factor of $\forceP_{\beta}$ of the form} 
    \\& \text{$\operatorname{Code} (x,y,m,k)$} \}
\end{align*}

Note that $R_{\eta}$ is a countable set of reals and can therefore be identified with a real itself, which we will do. The real $R_{\eta}$ is an error term and indicates the set of coding areas we must avoid when expecting correct codes, at least for the codes which contain $\dot{y}_{\eta},\dot{m}_{\eta}$ and $k_{\eta}$.

\end{itemize}
Assume that $\beta$ is odd. Assume that $G_{\beta}$ is a $\forceP_{\beta}$-generic filter over $W$. Let $x,y,m,k$ denote the evaluations of the names given to us at stage $\beta$ by $F$ using $G_{\beta}$. We let $m, k \in \omega$ be the G"odel indices corresponding to the $\Sigma^1_3$-formulas $\varphi_m(v_0,v_1)$ and $\varphi_k(v_0,v_1)$.
We turn to the forcing $\forceP(\beta)^{G_{\beta}}$ we want to define at stage $\beta$ in our iteration.
Again we distinguish several cases.
\begin{itemize}
\item[(A)]  Assume that $W[G_{\beta}]$ thinks that there is an $\alpha_{\beta}$-allowable forcing $\forceQ$ relative to $E_{\alpha_{\beta}}$ and some $F'$ such that 
\begin{align*}
\forceQ \Vdash 
\exists z (z\in A_m(y) \cap A_k(y)).
\end{align*}
Then we pick the $<$-least such forcing, where $<$ is some previously fixed well-order. We denote this forcing with $\forceQ_1$
and use 
\[\forceP(\beta)^{G_{\beta}}:= \forceQ_1.\]
We do not change $R_{\beta}$ at such a stage.

\item[(B)] Assume that (A) is not true.

\begin{itemize} 

\item[(i)] Assume however that there is an $\alpha_{\beta}$-allowable forcing $\forceQ$ in $W[G_{\beta}]$ with respect to $E_{\alpha_{\beta}}$ such that 
\begin{align*}
\forceQ \Vdash 
 x \in A_m(y).
\end{align*}
Then we set 
\[\dot{\forceQ}_{\beta}^{G_{\beta}}= \forceP(\beta)^{G_{\beta}}:= \hbox{Code} ({x}, {y}, m,k, 0,\eta),\]
where $\eta$ is the least ordinal
for which the freshness condition holds at stage $\beta$. 
In that situation, we enlarge the $E$-set as follows. We let $({y}, {m}, {k})=: ( {y}_{\alpha_{\beta}}, {m}_{\alpha_{\beta}}, {k}_{\alpha_{\beta}})$ and \[ E^{G_{\beta +1}}_{\alpha_{\beta}+1}:= E^{G_{\beta}}_{\alpha_{\beta}} \cup \{ ({y}, {m}, {k} ) \} .\]
Further, if we let $r_{\zeta}$ be the real which uniquely defines the generic filter for   $\hbox{Code} ({x}_{\zeta}, {y}_{\zeta}, m_{\zeta},k_{\zeta}, i_{\zeta},\eta_{\zeta})$ at stage $\zeta$ of the iteration which witnesses $\Psi_i(x_{\zeta},y_{\zeta},m_{\zeta}, k_{\zeta})$ of some quadruple $(x_{\zeta},y_{\zeta},m_{\zeta},k_{\zeta}$). Then we collect all the countably many such reals $r_{\zeta}$ we have generically added so far in our iteration up to stage $\beta$ and code them into one real $R$ and let
\[ R_{\alpha_{\beta}+1 }:= R . \] 
\item[(ii)] Assume that (i) is wrong, but there is an $\alpha_{\beta}$-allowable forcing $\forceQ$ with respect to $E_{\alpha_{\beta}}$  in $W[G_{\beta}]$ such that 
\begin{align*}
\forceQ \Vdash 
 x \in A_k(y).
\end{align*}
Then we set
\[\forceP(\beta)^{G_{\beta}}:= \hbox{Code} ({x}, {y}, m,k, 1,\eta).\]
Where $\eta$ is the least ordinal for which the freshness condition at $\beta$ holds.
In that situation, we enlarge the $E$-set as follows. We let the new $E^{G_{\beta+1}}$ value $({y}_{\alpha_{\beta}}, m_{\alpha_{\beta}}, k_{\alpha_{\beta}})$ be $ ( {y}, m, k)$  and \[E^{G_{\beta+1}}_{\alpha_{\beta}+1}:= E^{G_{\beta}}_{\alpha_{\beta}} \cup \{ ( {y}, m, k) \}.\]
Further, if we let $r_{\zeta}$ be the real which is added by the coding forcing $\hbox{Code} ({x}_{\zeta}, {y}_{\zeta}, m_{\zeta},k_{\zeta}, i_{\zeta},\eta_{\zeta})$ at stage $\zeta$ of the iteration which witnesses $\Psi_i (x_{\zeta},y_{\zeta},m,k)$. Then we collect all the countably many such reals $r_{\zeta}$ we have added so far in our iteration up to stage $\beta$ and put them into one set $R$ and let
\[ R^{G_{\beta+1}}_{\alpha_{\beta}+1 }:= R  . \] 

\item[(iii)] If neither (i) nor (ii) is true, then there is no $\alpha_{\beta}$-allowable forcing $\forceQ$ with respect to $E_{\alpha_{\beta}}$ which forces $x \in A_m(y)$ or $x \in A_k(y)$, and we set
\[ \forceP(\beta)^{G_{\beta}}:=\hbox{Code} ({x}, {y}, m,k, i,\eta),\]
for $i \in 2$. The choice of $i \in 2$ is dictated by the following reasoning: If there was an earlier stage $\beta' < \beta$ where $(x,y,m,k)$ was considered by $F$ and case B(i) applied there, then $i$ must be $0$. If there was an earlier stage $\beta' < \beta$ where $(x,y,m,k)$ was considered by the bookkeeping function $F$ and case B(ii) applied there, then $i$ must be 1. If at all earlier stages $\beta' < \beta$, whenever $(x,y,m,k)$ was considered by $F$ then neither case B(i) nor case B(ii) applied, or if $\beta$ is the first such stage itself  and neither case B(i)  nor case B(ii) applies then we force with the default forcing
\[ \forceP(\beta)^{G_{\beta}}:=\hbox{Code} ({x}, {y}, m,k, 0,\eta).\]

Further, if we let $r_{\zeta}$ be the real which is added by the coding forcing $\hbox{Code} ({x}_{\zeta}, {y}_{\zeta}, m_{\zeta},k_{\zeta}, i_{\zeta},\eta_{\zeta})$ at stage $\zeta \le \beta$ of the iteration which witnesses $\Psi_i$ of some quadruple $(x_{\zeta},y_{\zeta},m_{\zeta},k_{\zeta}$). Then we collect all the countably many such reals we have added so far in our iteration up to stage $\beta$ and put them into one set $R$ and let
\[ R^{G_{\beta+1}}_{\alpha_{\beta}+1 }:= R . \] 
Otherwise we force with the trivial forcing.
\end{itemize}
 \end{itemize}

\section{Towards a failure of $\Pi^1_3$-reduction}

Assume that $\beta < \omega_1$ is an even stage of our iteration. Our induction hypothesis includes that we have created already the iteration $\forceP_{\beta}^{G_{\beta}}$ up to stage $\beta$, and that we defined the notion of $\alpha_{\beta}$-allowable forcings for an ordinal $\alpha_{\beta} < \omega_1$. Assume that  $F(\beta) = (\dot{m},\dot{k},\dot{y})$, where $\dot{m},\dot{k}$ are $\forceP_{\beta}$-name for integers and $\dot{y}$ is a $\forceP_{\beta}$-name of a real number. We consider the following cases:
\subsection{Case 1}

We first assume that, working over $W[G_{\beta}]$, there is a further $\alpha_{\beta}$-allowable forcing $\forceP \in W[G_{\beta}]$ such that $\forceP$ adds two reals $x_0$ and $x_1$ and such that for a $\forceP$-generic filter $G$ over $W[G_{\beta}]$, $W[G_{\beta}] [G]$ satisfies:
\begin{enumerate}

\item $x_0 \in B_m(y) \setminus B_k(y)$ and $(x_0,y,a,b)$ is neither coded into $\vec{S}^0$ nor $\vec{S}^1$.
\item $x_1 \in B_k(y) \setminus B_m(y)$ and $(x_1,y,a,b)$ is neither coded into $\vec{S}^0$ nor $\vec{S}^1$.
\end{enumerate}

In that situation we first use the $<_W$-least $\forceP_{\beta}$-name of such a forcing $\forceP$ to add the reals $x_0$ and $x_1$.
After forcing with $\forceP$ we code $x_0$ and $x_1$ both into $\vec{S}^1$  via forcing with $\operatorname{Code} (x_0,a,b,1,\eta_0) \times \operatorname{Code} (x_1,a,b,1,\eta_1)$ for two distinct ordinals $\eta_0$ and $\eta_1$ both satisfying the freshness condition.

The two-step forcing \[\forceP \ast (\operatorname{Code} (x_0,a,b,1,\eta_0) \times \operatorname{Code} (x_1,a,b,1,\eta_1))= : \forceP^{G_{\beta}}(\beta)\] is the forcing we use at stage $\beta$ in the situation of case 1. We also change the notion of $\alpha_{\beta}$-allowable to $\alpha_{\beta}+1$-allowable which is defined to be $\alpha_{\beta}$-allowable together with the additional demand to neither use the forcing $\operatorname{Code} (x_0,a,b,0,\eta)$ nor $\operatorname{Code} (x_1,a,b,0,\eta)$ for any $\eta < \omega_1$. In other words we let $E_{\alpha_{\beta}+1}:= E_{\alpha_{\beta}} \cup  \{ (x_0,0) ,(x_1,0) \}$. Note that this choice ensures that $x_0$ and $x_1$ will both be elements of $B_0$ in all outer $\alpha_{\beta} +1$-allowable extensions.

The choice of $\forceP(\beta) =\dot{\forceQ}_{\beta}$ and $\alpha_{\beta} +1$-allowability will ensure that $B_m(y), B_k(y)$ can not reduce $B_0$ and $B_1$ in all possible $\alpha_{\beta}+1$-generic extensions of $L[G_{\beta+1}]$ as we shall see now.

\begin{lemma}
The sets $B_m(y)$ and $B_k(y)$ can not reduce $B_0$ and $B_1$ in all outer models of $W[G_{\beta+1}]$ which are obtained by a further $\alpha_{\beta}+1$-allowable forcing.
\end{lemma}
\begin{proof}
We shall consider three subcases to prove the lemma. We work in $M \supset L[G_{\beta+1}]$ where $M$ is an outer model obtained by a further $\alpha_{\beta}+1$-allowable forcing. 
\begin{itemize}

\item[Case 1a:]  First we assume that $x_0 \in B_m (y) \setminus B_k (y)$ and $x_1 \in B_k (y) \setminus B_m (y)$ still holds in $M$. In this situation neither $B_m (y) \subset B_1$ nor $B_k (y) \subset B_1$ can hold as witnessed by $x_0$ and $x_1$. In particular, $B_m (y)$ and $B_k (y)$ can not reduce $B_0$ and $B_1$. Note that this will remain true in all further outer models obtained  by an additional $\alpha_{\beta}+1$-allowable forcing as long as $x_0 \in B_m (y) \setminus B_k (y)$ and $x_1 \in B_k (y) \setminus B_m (y)$. If $x_i$ will drop out of $B_m (y)$ or $B_k (y)$ in some $\alpha_{\beta}+1$-allowable extension, then case 1b and case 1c will apply.
\item[Case 1b:] We assume that $x_0 \in B_m (y) \setminus B_k (y)$ but $x_1 \notin B_m (y) \cup B_k (y)$ holds in $M$. In this situation we can not have $B_0 \cup B_1 = B_m (y) \cup B_k (y)$ as $x_1 \in B_0$ and $x_1 \notin B_m (y) \cup B_k (y)$. Note that $x_1 \in B_0$ will remain true in outer $\alpha_{\beta} +1$-allowable models by our choice of $E_{\alpha_{\beta}+1}$ and $x_1 \notin B_m (y) \cup B_k (y)$ by the upwards absoluteness of the $\Sigma^1_3$-formulas $\lnot \psi_m$ and $\lnot \psi_k$. As a consequence, $B_m (y)$ and $B_k(y)$ can not reduce $B_0$ and $B_1$ in $M$.
\item[Case 1c:] In the dual case we assume that $x_1 \in B_m(y) \setminus B_k(y)$ but $x_0 \notin B_m(y) \cup B_k(y)$ holds in $M$ to derive, as above, that $B_0 \cup B_1 \ne B_m(y) \cup B_k(y)$.
\item[Case 1d:] If $x_0$ and $x_1$ are both not in $B_m(y)$ and $B_k(y)$, then again $B_m(y) \cup B_k(y) \ne B_0 \cup B_1$.

\end{itemize}

\end{proof}

\subsection{Case 2}
In the second case, we assume that case 1 does not apply. As a consequence, whenever we work over $W[G_{\beta}]$ and apply a further $\alpha_{\beta}$-allowable forcing $\forceP$ which adds two reals $x_0\ne x_1$ and does neither have $\operatorname{Code} (x_0,a,b,i,\eta)$ nor $\operatorname{Code} (x_1,a,b,i,\eta)$ for some $\eta< \omega_1$ and for $i \in \{ 0,1\}$ as a factor in the iteration, then $x_0$, $x_1$ will not satisfy that $x_0 \in B_m(y) \setminus B_k(y)$ and $x_1 \in B_k(y) \setminus B_m(y)$.

Again, we shall split into subcases:

\begin{enumerate}
\item[Case 2a:] There is an $\alpha_{\beta}$-allowable forcing $\mathbb{P}$ and a $G\subset \mathbb{P}$ such that in $W[G_{\beta}] [G]$ there are $x_0\ne x_1$ such that $x_0,x_1 \in B_m (y)\setminus B_k (y) $. In this situation, we force over $W[G_{\beta}] [G]$ with $\operatorname{Code} (x_0,a,b,1,\eta_0) \times \operatorname{Code} (x_1,a,b,0,\eta_1)$ for two ordinals $\eta_0,\eta_1$ for which the freshness condition holds. Let $H_0 \times H_1 \subset \operatorname{Code} (x_0,a,b,1,\eta_0) \times \operatorname{Code} (x_1,a,b,0,\eta_1)$ be a $W[G_{\beta}][G]$-generic filter, and define $G_{\beta+1}:= G_{\beta} \ast G \ast (H_0 \times H_1)$. 

We also define $\alpha_{\beta} +1$-allowable by enlarging the restricted set to $E_{\alpha_{\beta}+1}:= E_{\alpha_{\beta}} \cup \{ (x_0,0), (x_1,1) \}$. Formally, this introduces the additional rule that for any $\alpha_{\beta}+1$-allowable iteration $\mathbb{Q} = (\mathbb{Q}_{\gamma} \mid \gamma < \delta)$ and any stage $\gamma < \delta$, the iterand must satisfy:
$$ \forall \zeta < \omega_1 \Big( \mathbb{Q}(\gamma)^{G_{\gamma}} \neq \operatorname{Code}(x_0,a,b,0,\zeta) \land \mathbb{Q}(\gamma)^{G_{\gamma}} \neq \operatorname{Code}(x_1,a,b,1,\zeta) \Big) $$
Note that as a consequence of this restriction, $x_0 \in B_0$ and $x_1 \in B_1$ will remain true in all $\alpha_{\beta} +1$-allowable generic extensions of $W[G_{\beta+1}]$.
Then we argue as follows:
\begin{lemma}
Let $M$ be an outer model of $W[G_{\beta+1}]$ obtained via
an $\alpha_{\beta}+1$-allowable forcing. Then $B_m(y)$ and $B_k(y)$ can not reduce $B_0$ and $B_1$ over $M$.
\end{lemma}

\begin{proof}
We split into several cases.
Assume first that in $M \supset W[G_{\beta+1}]$, $x_0,x_1 \in B_m(y)$ still holds true. Then, as $x_0 \in B_0 \setminus B_1$ and $x_1 \in B_1 \setminus B_0$, the set $B_m(y)$ can neither reduce $B_0$ nor $B_1$.  Indeed $x_0$ witnesses that $B_m (y) $ is not a subset of $B_1$ and $x_1$ witnesses that $B_m(y)$ is not a subset of $B_0$.  If, in a further $\alpha_{\beta}+1$-allowable extension of $M \supset W[G_{\beta+1}]$, $x_0 \notin B_m(y)$ or $x_1 \notin B_m(y)$ then this will be dealt with in the next subcases.

If in $M \supset W[G_{\beta+1}]$, $x_0 \notin B_m(y)$, yet $x_1 \in B_m(y)$ then $x_0 \notin B_m(y) \cup B_k(y)$ and $x_0 \in B_0$ in all $\alpha_{\beta}+1$-allowable generic extensions of $M$. So $B_m(y) \cup B_k(y) \ne B_0 \cup B_1$ and $B_m(y), B_k(y)$ can not reduce $B_0, B_1$. The same argument works if $x_1 \notin B_m(y)$ and $x_0 \in B_m(y)$.

If in $M$, $x_0, x_1 \notin B_m(y)$ anymore, then again $B_m(y),B_k(y)$ can not reduce $B_0,B_1$.

\end{proof}

\item[Case 2b:] The dual situation in which we can force $x_0 \ne x_1$ such that $x_0, x_1 \in B_k(y) \setminus B_m(y)$ is dealt with in the analogous way.

\item[Case 2c:] If neither 2a nor 2b are true, then one cannot force two distinct reals $x_0$, $x_1$ into $B_m(y) \setminus B_k(y)$ with $\alpha_{\beta}$-allowable forcings which do not contain $\operatorname{Code} (x_0,a,b,i,\eta)$ or $\operatorname{Code} (x_1,a,b,i,\eta)$ for $i \in \{0,1\}$ and $\eta <\omega_1$. And the same holds true for $B_k(y) \setminus B_m(y)$. 

But it is straightforward to use an $\alpha_{\beta}$-allowable forcing $\forceP$ over $W[G_{\beta}]$ which will add $\aleph_1$-many new reals $(z_i \mid i < \omega_1)$ and for any $\eta <\omega_1$ neither $\operatorname{Code} (z_i,a,b,0,\eta)$ nor $\operatorname{Code} (z_i,a,b,1,\eta)$ is a factor of $\forceP$\footnote{Indeed every coding forcing $\operatorname{Code}(x',y',m_0,k_0,0,\eta')$ for $m_0,k_0$ being G\"odel numbers of two $\Sigma^1_3$-formulas which are always false, for arbitrary reals $x'$ and $y'$ and an $\eta'$ which satisfies the freshness condition (note that using such a forcing is always permitted by the rules of allowability) will add $\aleph_1$-many new reals which are as desired.}.
    
We force with such a $\forceP$, let $G \subset \forceP$ be generic over $W[G_{\beta}]$, and let $W[G_{\beta+1}]=W[G_{\beta}] [G]$. We define $\alpha_{\beta} +1$-allowable forcing by enlarging the restricted set $E_{\alpha_{\beta}+1} := E_{\alpha_{\beta}} \cup \{ (z_i, j) \mid i < \omega_1, j \in \{0, 1\} \}$. Formally, this introduces the additional rule that for any future $\alpha_{\beta}+1$-allowable iteration $\mathbb{Q} = (\mathbb{Q}_{\gamma} \mid \gamma < \delta)$ and any stage $\gamma < \delta$, the iterand must satisfy:
$$ \forall i < \omega_1 \, \forall j \in \{0, 1\} \, \forall \eta < \omega_1 \Big( \mathbb{Q}(\gamma)^{G_{\gamma}} \neq \operatorname{Code}(z_i, a, b, j, \eta) \Big). $$
The upshot of this is that now the following holds.

\begin{lemma}
If $M$ is an outer model of $W[G_{\beta+1}]$, obtained with a further $\alpha_{\beta}+1$-allowable forcing. Then $B_m(y)$ and $B_k (y)$ can not reduce $B_0,B_1$.

\end{lemma}
\begin{proof}
 For every $\alpha_{\beta}+1$-allowable extension $M$ of $W[G_{\beta+1}]$, $\forall i < \omega_1 (z_i \in B_0 \cup B_1)$ must hold in $M$, yet for any pair $z_i \ne z_j$,
 $M \models z_i, z_j \in B_m(y) \cap B_k (y)$ or $M \models z_i,z_j \notin (B_m(y) \cup B_k(y) ) $. In both cases, $z_i, z_j$ witness that $B_m(y)$ and $B_k(y)$ can not reduce $B_0,B_1$ over $M$.

\end{proof}

 To summarize: in both cases we defined an extension $W[G_{\beta+1} ]$ of $W[G_{\beta}]$, and the notion of $\alpha_{\beta}+1$-allowable forcings. Additionally we found reals which witness that $B_m(y)$ and $B_k(y)$ can not reduce $B_0$ and $B_1$ in all further outer models $M$ of $W[G_{\beta+1}]$ which are obtained with a further $\alpha_{\beta}+1$-allowable forcing. 
\end{enumerate}

At limit stages $\beta$, we use the finite support to define the limit partial order $\forceP_{\beta}$ and set $E_{\alpha_{\beta}}= \bigcup_{\eta < \beta}  E_{\alpha_{\eta}}$.
This ends the definition of $\forceP_{\omega_1}$.

\section{Discussion of the resulting universe}
We let $G_{\omega_1}$ be a $\forceP_{\omega_1}$-generic filter over $W$.
As a first note we observe that for every real $r$ in $W[G_{\omega_1}]$ there is a $\beta < \omega_1$ such that $r \in W[G_{\beta}]$ by standard facts about finite support iterations of ccc forcings. 
In particular, if a $\Sigma^1_3$-statement is true in $W[G_{\omega_1}]$ it will already be true in some $W[G_{\beta}]$.
As $W[G_{\omega_1}]$ is a proper extension of $W$, $\omega_1$ is preserved, in fact all cardinals in $L$ are preserved but we will not need this. Moreover $\CH$ remains true by Lemma \ref{FirstPropertiesOfAllowableForcings}.

A second observation is that for every stage $\beta$ of our iteration and every $\eta > \beta$, the intermediate forcing $\forceP_{[\beta, \eta)}$, defined as the factor forcing of $\forceP_{\beta}$ and $\forceP_{\eta}$, is always an $\alpha_{\beta}$-allowable forcing relative to $E_{\alpha_{\beta}}$ and some bookkeeping. This is intuitively clear as by the definition of the iteration, we force at every stage $\beta$ with a $\alpha_{\beta}$-allowable forcing  relative to $E_{\alpha_{\beta}}$ and $\alpha_{\beta}$-allowable becomes a stronger notion as we increase $\alpha_{\beta}$.

\begin{lemma}
For every stage $\beta < \omega_1$ of the iteration and every $\eta$ such that $\beta < \eta < \omega_1$, the intermediate forcing $\mathbb{P}_{[\beta,\eta)}$, defined as the factor forcing such that $\mathbb{P}_{\eta} \cong \mathbb{P}_{\beta} \ast \mathbb{P}_{[\beta,\eta)}$, is an $\alpha_{\beta}$-allowable forcing relative to $E_{\alpha_{\beta}}$ and some bookkeeping function $F'$.
\end{lemma}

\begin{proof}
First, he bookkeeping function $F'$ is the canonical projection of $F$ over $W[G_{\beta}]$. Specifically, if the original bookkeeping function evaluated at stage $\xi \ge \beta$ yields a tuple of $\mathbb{P}_{\xi}$-names $F(\xi) = (\dot{x}, \dot{y}, \dot{m}, \dot{k}, \dot{i}, \dot{\gamma})$, then $F'(\xi - \beta)$ is defined as the tuple of $\mathbb{P}_{[\beta, \xi)}$-names obtained by evaluating those original names using the $\mathbb{P}_{\beta}$-generic filter $G_{\beta}$ (i.e., mapping $\dot{\tau} \mapsto \dot{\tau}^{G_{\beta}}$). Consequently, $F'$ systematically enumerates the exact same intended targets over the intermediate model $W[G_{\beta}]$ that $F$ enumerates over $W$.
We fix $\beta$ and proceed by induction on $\eta > \beta$.

We start with the base case $\eta = \beta + 1$. 
By the definition of the iteration at stage $\beta$, the forcing $\mathbb{P}(\beta)$ (which constitutes $\mathbb{P}_{[\beta, \beta+1)}$) is chosen precisely to be an $\alpha_{\beta}$-allowable forcing relative to $E_{\alpha_{\beta}}$. Thus, the base case trivially holds.

Now we argue for the successor step $\eta \to \eta + 1$.
Assume the inductive hypothesis holds for $\eta$, meaning $\mathbb{P}_{[\beta,\eta)}$ is $\alpha_{\beta}$-allowable with respect to $E_{\alpha_{\beta}}$. At stage $\eta$, the rules of our iteration dictate that we force with $\mathbb{P}(\eta)$, which is chosen to be an $\alpha_{\eta}$-allowable forcing relative to $E_{\alpha_{\eta}}$. 

Because the sequence of ordinals $\alpha_{\xi}$ is non-decreasing along the iteration, we know that $\alpha_{\eta} \ge \alpha_{\beta}$. Consequently, $E_{\alpha_{\beta}} \subseteq E_{\alpha_{\eta}}$. As noted in the remarks on allowability, if a forcing is $\delta_1$-allowable with respect to a set of restrictions $E_{\delta_1}$, it is necessarily $\delta_2$-allowable with respect to the smaller set of restrictions $E_{\delta_1} \restriction \delta_2$ for any $\delta_2 \le \delta_1$. 

Applying this principle, since $\mathbb{P}(\eta)$ is $\alpha_{\eta}$-allowable with respect to $E_{\alpha_{\eta}}$, it is inherently $\alpha_{\beta}$-allowable with respect to $E_{\alpha_{\beta}}$. Because both $\mathbb{P}_{[\beta,\eta)}$ and the iterand $\mathbb{P}(\eta)$ obey the rules of $\alpha_{\beta}$-allowability with respect to $E_{\alpha_{\beta}}$, their concatenation $\mathbb{P}_{[\beta,\eta+1)} \cong \mathbb{P}_{[\beta,\eta)} \ast \mathbb{P}(\eta)$ is also an $\alpha_{\beta}$-allowable forcing relative to $E_{\alpha_{\beta}}$.

Finally we argue for the limit step when $\eta$ is a limit ordinal.
Assume $\eta > \beta$ is a limit ordinal and that for all $\xi \in (\beta, \eta)$, the intermediate forcing $\mathbb{P}_{[\beta,\xi)}$ is $\alpha_{\beta}$-allowable with respect to $E_{\alpha_{\beta}}$. 

By the definition of our iteration, $\mathbb{P}_{[\beta,\eta)}$ is the finite support limit of the sequence $\langle \mathbb{P}_{[\beta,\xi)} \mid \beta \le \xi < \eta \rangle$. By the definition of allowability at limit stages, a finite support iteration of allowable forcings remains allowable under those exact same parameters. Therefore, $\mathbb{P}_{[\beta,\eta)}$ is $\alpha_{\beta}$-allowable with respect to $E_{\alpha_{\beta}}$ and some bookkeeping function $F'$. 

This completes the induction.
\end{proof}

As a third observation we note that the ${\Pi}^1_3$-reduction property fails in $W[G_{\omega_1}]$.

\begin{lemma}
In $W[G_{\omega_1}]$, the sets $B_0$ and $B_1$ cannot be reduced by any pair of boldface $\Pi_3^1$ sets. That is, for any parameter $y \in \omega^\omega \cap W[G_{\omega_1}]$ and any $m, k \in \omega$, the sets $B_m(y)$ and $B_k(y)$ do not reduce $B_0$ and $B_1$.
\end{lemma}

\begin{proof}
Let $y \in W[G_{\omega_1}]$ be an arbitrary real, and let $m, k \in \omega$ be indices for $\Pi_3^1$ formulas. There is some $\gamma < \omega_1$ such that $\dot{y}$ is a valid $\mathbb{P}_\gamma$-name.

By the definition of our bookkeeping function $F$, it lists every tuple in $H(\omega_1)$ cofinally often on both even and odd ordinals. Therefore, there exists an even stage $\beta > \gamma$ such that $F(\beta) = (m, k, \dot{y})$ and  $\dot{y}$ can be evaluated with the generic filter $G_\beta$, meaning $y = \dot{y}^{G_\beta}$ is completely determined in $W[G_\beta]$.

At this even stage $\beta$, the construction of the iteration explicitly works towards the failure of $\Pi_3^1$-reduction for the specific pair $B_m(y)$ and $B_k(y)$. Depending on the forcings available in $W[G_\beta]$, the iteration strictly follows one of the subcases (Case 1, Case 2a, Case 2b, or Case 2c) detailed in Section 9.

As shown in Lemmas 9.1, 9.2, and 9.3, we ensured that in any outer model $M \supset W[G_{\beta+1}]$ obtained via a further $\alpha_\beta+1$-allowable forcing, the sets $B_m(y)$ and $B_k(y)$ cannot reduce $B_0$ and $B_1$.

It remains to show that the final model $W[G_{\omega_1}]$ is exactly such an outer model. The remainder of the iteration, $\mathbb{P}_{[\beta+1, \omega_1)}$, is a finite support iteration of allowable forcings. As shown in the preceding lemma, for any $\xi > \beta$, the intermediate forcing $\mathbb{P}_{[\beta+1, \xi)}$ is an $\alpha_{\beta+1}$-allowable forcing relative to $E_{\alpha_{\beta+1}}$. Because the allowability constraints are strictly upward-preserving (meaning $E_{\alpha_\beta+1} \subseteq E_{\alpha_{\beta+1}}$), any $\alpha_{\beta+1}$-allowable forcing inherently satisfies the strict constraints of $\alpha_\beta+1$-allowability introduced at stage $\beta$. 

Therefore, the tail forcing $\mathbb{P}_{[\beta+1, \omega_1)}$ is an $\alpha_\beta+1$-allowable forcing over $W[G_{\beta+1}]$. Consequently, $W[G_{\omega_1}]$ is an $\alpha_\beta+1$-allowable generic extension of $W[G_{\beta+1}]$. Applying Lemmas 9.1, 9.2, and 9.3, the assertion follows, and we conclude that $B_m(y)$ and $B_k(y)$ fail to reduce $B_0$ and $B_1$ in $W[G_{\omega_1}]$.
\end{proof}

We still need to show that in $W[G_{\omega_1}]$ the $\bf{\Sigma}^1_3$-separation property is true. 
For a pair of disjoint $\Sigma^1_3(y)$-sets, $A_m(y)$ and $A_k(y)$,
we consider the least stage $\beta$ such that there is a $\forceP_{\beta}$-name $\dot{z}$ such that $\dot{z}^{G_{\beta}}=z$ and $(z,y,m,k)$ are considered by $F$ at stage $\beta$. Recall that the generic extension $W[G_{\beta}]$ can be written as $W[R_{\alpha_{\beta}}]$ for  a real $R_{\alpha_{\beta}}$ which codes the countably many, generically added reals $r_i$ we created every time we forced with a forcing of the form $\operatorname{Code} (x,y,m,k,l,\eta)$  which serve as witnesses to the formulas
$\Psi_l(r_i, (x,y,m,k))$.

As mentioned already, the role of $R_{\alpha_{\beta}}$ is that of an error term. We might have added false patterns in our iteration so far, but these false patterns will appear in a coded form in $R_{\alpha_{\beta}}$.  Our definitions will ensure that for the pair $A_m(y)$ and $A_k(y)$, modulo $R_{\alpha_{\beta}}$, the set of reals $x$ for which the quadruple $(x,y,m,k)$ is coded into $\vec{S}^0$, and the set of reals $x$, for which the quadruple $(x,y,m,k)$ is  coded into $\vec{S}^1$ will separate $A_m(y)$ and $A_k(y)$:

\begin{align*}
x \in D^0_{y,m,k} (R_{\alpha_{\beta}})\Leftrightarrow \exists r (&L[r, R_{\alpha_{\beta}}]  \models (x,y,m,k) \text{ can be read off from a code} \\& \text{written on an } \omega_1\text{-many $\omega$-blocks of elements of } \\& \vec{S^0} \text{ and the coding area of $(x,y,m,k)$} \\& \text{is almost disjoint from each coding area in $R_{\alpha_{\beta}}$} ).
\end{align*}
and
\begin{align*}
x \in D^1_{y,m,k} (R_{\alpha_{\beta}})\Leftrightarrow \exists r(&L[r, R_{\alpha_{\beta}}]  \models (x,y,m,k) \text{ can be read off from a code} \\& \text{written on an } \omega_1\text{-many $\omega$-blocks of elements of } \\& \vec{S^1} \text{ and the coding area of $(x,y,m,k)$} \\& \text{is almost disjoint from each coding area in $R_{\alpha_{\beta}}$} ).
\end{align*}

\begin{lemma}
In $W[G_{\omega_1}]$, if $y \in \omega^{\omega}$ is an arbitrary  parameter, $R$ a real and $m,k$ natural numbers, then the sets $D^0_{y,m,k}(R)$ and $D^1_{y,m,k}(R)$ are $\Sigma^1_3(R)$-definable.
\end{lemma}
\begin{proof}
The proof is a standard calculation using the obvious modification of the formula $\Phi$ employing $R$ as the set which codes the coding areas a real must avoid to be in $D^0_{y,m,k} (R)$ or $D^1_{y,m,k}(R)$.
\end{proof}
It follows  from the definition of the iteration that for any  real parameter $y$ and any $m,k \in \omega$, $D^0_{y,m,k} (R_{\alpha_{\beta}})  \cup D^1_{y,m,k} (R_{\alpha_{\beta}}) = \omega^{\omega}$, at least as soon as $A_m(y) \cap A_k(y) =\emptyset$ in $W[G_{\omega_1}]$. Indeed, as our bookkeeping function visits each tuple of a reals in our iteration $\aleph_1$-many times, the tuple $(x,y,m,k)$  will be considered by the bookkeeping unboundedly often. Let $\beta$ be the least such stage. As for this tuple, case (B) (i), (ii) or (iii) applies, the tuple gets coded into $\vec{S}^0$ or $\vec{S}^1$ at stages $\beta' > \beta$, which gives that $D^0_{y,m,k} (R_{\alpha_{\beta}})  \cup D^1_{y,m,k} (R_{\alpha_{\beta}}) = \omega^{\omega}$.
The next lemma establishes that the sets are indeed separating.
\begin{lemma}
In $W[G_{\omega_1}]$, let $y$ be a real and let $m,k \in \omega$ be such that $A_m(y) \cap A_k(y)=\emptyset.$ Then there is an real $R$ such that the sets  $D^0_{y,m,k}(R)$ and $ D^1_{y,m,k}(R)$ partition the reals. 
\end{lemma}
\begin{proof}
Let $\beta$ be the least stage such that there is a real $x$ such that $F(\beta) =(\dot{x}, \dot{y},\dot{m},\dot{k})$ with $\dot{x}^G_{\beta}=x$, $\dot{y}^G_{\beta}=y$. Let $R$ be $R_{\alpha_{\beta}}$ for $R_{\alpha_{\beta}}$ being defined as above. Then, as $A_m(y)$ and $A_k(y)$ are disjoint in $W[G_{\omega_1}]$, by the rules of the iteration, case B must apply at $\beta$. 

Assume now for a contradiction, that $D^0_{y,m,k}(R)$ and $ D^1_{y,m,k}(R)$ do have non-empty intersection in $W[G_{\omega_1}]$. Let $z \in D^0_{y,m,k}(R) \cap D^1_{y,m,k}(R)$ and let $\eta > \beta$ be the first stage of the iteration which contains the witnessing reals for the two $\Sigma^1_3$-sets $D^0_{y,m,k}(R)$ and $D^1_{y,m,k}(R)$ which confirm that $z$ is in the intersection of these two sets.

Then, by the rules of the iteration and as there is a $z \in D^0_{y,m,k}(R) \cap D^1_{y,m,k}(R)$ we must have forced with $\operatorname{Code} (z,y,m,k,1,\eta)$ and with $\operatorname{Code} (z,y,m,k,0,\eta)$ in our iterations at some stages. We assume first that we must have used case B(i) at the first stage  $\delta \ge \beta$ of the iteration where $(z,y,m,k)$ is listed by the bookkeeping $F$, and case B(ii) at  a later stage $\gamma < \eta$. But this would imply, that at stage $\delta$, there is an $\alpha_{\delta}$-allowable forcing  $\forceQ_{\delta}$ with respect to $E_{\alpha_{\delta}}$, which forces $z \in A_m(y)$, yet at stage $\gamma > \delta$, there is an $\alpha_{\gamma}$-allowable forcing which forces $z \in A_k(y)$.
As $\gamma> \delta$, the $\alpha_{\gamma}$-allowable forcing $\forceQ_{\gamma}$ which witnesses that we are in case B (ii) at stage $\gamma$ in our iteration, is also $\alpha_{\delta}$-allowable. But this means that, over $W[G_{\delta}]$, the intermediate forcing $\forceP_{\delta,\gamma}$, which is also $\alpha_{\delta}$-allowable can be extended to the $\alpha_{\delta}$-allowable forcing which first uses $\forceP_{\delta,\gamma}$ and then $\forceQ_{\gamma}$  yielding an $\alpha_{\delta}$-allowable forcing which forces $z \in A_k(y)$. Note here that, as in the proof of Lemma \ref{productallowable} we can shift the coding areas of $\forceQ_{\gamma}$, and therefore guarantee the crucial condition $C^{\forceP_{\delta,\gamma}} \cap C^{\forceQ_{\gamma}}= \emptyset.$

On the other hand, in $W[G_{\delta}]$ we do have $\forceQ_{\delta}$ which forces $z \in A_m(y)$, as we assumed that at stage $\delta$ we are in case B (i). Note that again, we can shift the coding areas for this $\forceQ_{\delta}$ at will. 
Now the product $\forceQ_{\delta} \times (\forceP_{\delta, \gamma} \ast \, \forceQ_{\gamma})$ is $\alpha_{\delta}$-allowable, and by upwards absoluteness of $\Sigma^1_3$-formulas we get that 
\[ \forceQ_{\delta} \times (\forceP_{\delta,\gamma} \, \ast \, \forceQ_{\gamma}) \Vdash z \in A_m(y) \cap A_k(y). \]
But this would mean that at stage $\delta$, we are in case A in the definition of our iteration, which is a contradiction. This finishes the cases where B(i) was first applied when the bookkeeping listed $(z,y,m,k)$ for the first time and then case B(ii) was applied at a later stage when the bookkeeping revisited $(z,y,m,k)$. The case where first B(ii) was applied and then B(i) leads to a contradiction in an analogous way.

If we assume that we were in case B(i) at the first stage $\delta \ge \beta$ of the iteration where $(z,y,m,k)$ is listed by the bookkeeping $F$, and case B(iii) at a later stage $\gamma < \eta$ then by the definition of the iteration we have to force at both stages with $\operatorname{Code} (z,y,m,k,0,\eta)$ and with $\operatorname{Code} (z,y,m,k,0,\eta')$ for two distinct ordinals $\eta,\eta'$. This is a contradiction to our assumption that $z \in  D^0_{y,m,k}(R) \cap D^1_{y,m,k}(R)$.

In the case where we first applied case B(ii) and then later B(iii), we argue similarly to the above paragraph.

Finally if we were in case B(iii) first when the bookkeeping hit $(z,y,m,k)$ for the first time, and later in the iteration at case B(ii) when the bookkeeping hit $(z,y,m,k)$ again, then we argue that this is impossible as well. Indeed as the set of forcings which are $\alpha_{\beta}$-allowable shrinks as we increase $\alpha_{\beta}$, being in case B(iii) first and then in B(ii) at a later stage is impossible as the forcing $\forceQ$ witnessing that we are in case B(ii) would also witness that we can not be in case B(iii) at an earlier stage as there case B(ii) has to apply as well.

This ends the discussion of all possible cases and proves the lemma.

\end{proof}

\begin{lemma}
In $W[G_{\omega_1}]$, for every pair $m,k$ and every parameter $y \in \omega^{\omega}$ such that $A_m(y) \cap A_k(y) = \emptyset$ there is a real $R$ such that 
\[ A_m(y) \subset D^0_{y,m,k}(R) \land A_k(y) \subset D^1_{y,m,k}(R)\]
\end{lemma}
\begin{proof}
The proof is by contradiction. Assume that $m,k$ and $y$ are such that for every real $R$ there is $z$ such that $z \in A_m(y) \cap D^1_{y,m,k}(R)$ or $z \in A_k(y) \cap D^1_{y,m,k} (R)$.
We consider the smallest ordinal $\beta < \omega_1$ such that  $F(\beta)$ lists a quadruple of the form $(x,y,m,k)$ for which $W[G_{\omega_1}] \models x \in A_m(y) \cap D^1_{y,m,k}$ and let $R= R_{\alpha_{\beta}}$. As $A_m(y)$ and $A_k(y)$ are disjoint we know that at stage $\beta$ we were in case B. 
As $x$ is coded into $\vec{S^1}$ after stage $\beta$ and by the last Lemma, Case B(i) is impossible at $\beta$. Hence, without loss of generality we may assume that case B(ii) applies at $\beta$. As a consequence, there is a forcing $\forceQ_2 \in W[G_{\beta} ]$ which is $\alpha_{\beta}$-allowable with respect to $E_{\alpha_{\beta}}$ which forces $\forceQ_2 \Vdash x \in A_k(y)$. Note that in that case we collect all the reals which witness $( {\ast} {\ast} {\ast})$ for some quadruple to form the set $R_{\alpha_{\beta}}$. 

As $x \in A_m(y) \cap  D^1_{y,m,k} (R)$, we let $\beta'> \beta$ be the first stage such that $W[G_{\beta'}] \models x \in A_m(y)$. By Lemma \ref{productallowable}, $W[G_{\beta}]$ thinks that $\forceQ_2 \times \forceP_{\beta \beta'}$ is $\alpha_{\beta}$-allowable with respect to $E_{\alpha_{\beta}}$, yet 
$\forceQ_2 \times \forceP_{\beta \beta'} \Vdash x \in A_m(y) \cap A_k(y)$. Thus, at stage $\beta$, we must have been in case A. This is a contradiction.
\end{proof}

The next lemma will finish the proof of our theorem:

\section{Lifting to $M_n$}

The ideas presented can be used to construct a universe in which $\bf{\Sigma}^1_{n+3}$-separation can be separated from $\Pi^1_{n+3}$-reduction. We just state it without an attempt to prove it.
\begin{theorem}
Assuming that $M_n$ exists, there is a model of $\bf{\Sigma}^1_{n+3}$-separation over which there is a pair of $\Pi^1_{n+3}$-sets, which can not be reduced by any pair of $\bf{\Pi}^1_{n+3}$-sets.
\end{theorem}
Its proof is a direct application of the ideas from \cite{Ho2} which can be used to translate the argument for the third level of the projective hierarchy to $M_n$.
As there are no new ideas needed, and the translation works in a very similar manner to \cite{Ho2}, we will not go into any details here.

\section{Open questions}
We end with several questions which are related to this article.
\begin{question}
Does there exist a universe in which $\bf{\Sigma}^1_3$-separation holds, but $\Pi^1_n$-reduction fails for any $n \ge 3$?
\end{question}
Note for this question that in our construction of the failure of $\Pi^1_3$-reduction, we used Shoenfield absoluteness, or rather the upwards absoluteness of $\Sigma^1_3$-formulas, several times. Thus a failure of $\Pi^1_n$-reduction would need new ideas and arguments.

\begin{question}
Can one force a universe of $\bf{\Sigma}^1_4$-separation in which $\Pi^1_4$-reduction fails over just $L$?
\end{question}

By a classical result of Novikov, for any projective pointclass $\Gamma$, it is impossible to have $\Gamma$ and $\check{\Gamma}$-reduction simultaneously. The case  for separation is still unknown.
\begin{question}
Can one force a universe where $\Sigma^1_3$- and $\Pi^1_3$-separation hold simultaneously?
\end{question}

\bibliographystyle{plain} 
\bibliography{references}

\end{document}